\newif\ifsiam     
\newif\ifnummat   
\newif\ifcvs      
\newif\ifmoc      
\newif\ifanm      
\newif\ifjcam     
\newif\ifjcm      

\siamfalse
\nummatfalse
\cvsfalse
\mocfalse
\anmfalse
\jcamfalse
\jcmfalse
 
\nummattrue

\ifsiam
    \documentclass[review]{siamart1116}
    \newtheorem{remark}{Remark}[section]
\fi

\ifnummat
    \RequirePackage{fix-cm}
    \documentclass[smallextended]{svjour3}
    \smartqed
\fi

\ifcvs
    \documentclass[twocolumn,fleqn,runningheads,final]{svjour3}
    \journalname{Computing and Visualization in Science}
    \usepackage{mathptmx}
    \usepackage{mathptmx}
    \DeclareMathAlphabet{\mathcal}{OMS}{cmsy}{m}{n}
\fi

\ifmoc
    \documentclass{mcom-l}
    \newtheorem{theorem}{Theorem}[section]
    \newtheorem{lemma}[theorem]{Lemma}
    \newtheorem{corollary}[theorem]{Corollary}

    \theoremstyle{definition}
    \newtheorem{definition}[theorem]{Definition}

    \theoremstyle{remark}
    \newtheorem{remark}[theorem]{Remark}

    \numberwithin{equation}{section}

\fi

\ifanm
    \documentclass[ ]{elsarticle}
    \journal{Applied Numerical Mathematics}

    \newtheorem{theorem}{Theorem}
    \newtheorem{lemma}[theorem]{Lemma}
    \newtheorem{corollary}[theorem]{Corollary}
    \newtheorem{remark}{Remark}
\fi

\ifjcam
    \documentclass[ ]{elsarticle}
    \journal{Journal of Computational and Applied Mathematics}

    \newtheorem{theorem}{Theorem}
    \newtheorem{lemma}[theorem]{Lemma}

\fi

\ifjcm
    \documentclass{jcmlatex}

\fi

\usepackage{bm,amssymb,amsmath}

\ifsiam
    \usepackage[all]{xy}
    \usepackage{braket,amsfonts,amsopn}
    \usepackage{makecell,multirow,diagbox}
    \usepackage{mathrsfs}
    \usepackage{graphicx,epstopdf}
    \usepackage{subfigure}
    \numberwithin{equation}{section}
    \numberwithin{theorem}{section}
\else
    \usepackage[all]{xy}
    \usepackage[colorlinks]{hyperref}
    \usepackage{multirow,subfigure}
    \usepackage{graphicx,epstopdf}

    \numberwithin{lemma}{section}
    \numberwithin{theorem}{section}
    \numberwithin{definition}{section}
    \numberwithin{corollary}{section}
    \numberwithin{equation}{section}
\fi

\DeclareMathOperator{\divg}{div}
\DeclareMathOperator{\grad}{grad}

\newcommand{\Th}{\mathcal{T}_{h}}

\newcommand{\CV}{\mathcal{V}}

\newcommand{\CB}{\mathcal{B}}

\newcommand{\CN}{\mathcal{N}}
\newcommand{\CE}{\mathcal{E}}
\newcommand{\CP}{\mathcal{P}}
\newcommand{\DD}{\mathcal{D}}
\newcommand{\CQ}{\mathcal{Q}}

\newcommand{\Eh}{\mathcal{E}_{h}}
\newcommand{\Nh}{\mathcal{N}_{h}}

\newcommand{\dd}[1]{\partial_{\bm{d}}#1}
\newcommand{\dtk}[1]{\partial_{\bm{t}_{k}}#1}
\newcommand{\dtkp}[1]{\partial_{\bm{t}_{k+1}}#1}
\newcommand{\dtkm}[1]{\partial_{\bm{t}_{k-1}}#1}
\newcommand{\dnk}[1]{\partial_{\bm{n}_{k}}#1}

\newcommand{\dtnk}[1]{\partial^2_{\bm{t}_k\bm{n}_{k}}#1}
\newcommand{\dte}[1]{\partial_{\bm{t}_{e}}#1}

\def\rebAuthor{Randolph E. Bank}

\def\yulAuthor{Yuwen Li}

\ifsiam
    \def\rebShortAuthor{R.~E.~Bank}
    \def\yulShortAuthor{Y.~Li}
\else
    \def\rebShortAuthor{R. E. Bank}
    \def\yulShortAuthor{Y. Li}
\fi

\def\rebAddress{Department of Mathematics, University of California, San Diego,
 La Jolla, California 92093-0112.}

\def\yulAddress{Department of Mathematics, University of California, San Diego,
 La Jolla, California 92093-0112.}

\def\rebEmail{rbank@ucsd.edu}
\def\yulEmail{yul739@ucsd.edu}

\def\rebThanks{The work of this author was supported by the National
Science Foundation under contracts DMS-1318480.}

\def\yulThanks{}

\title{Superconvergent recovery of Raviart--Thomas mixed finite elements on triangular grids}

\def\shortTitle{Superconvergent recovery of Raviart--Thomas elements}

\def\myKeywords{superconvergence, mildly structured grids, mixed methods, Raviart--Thomas elements, second order elliptic equations}

\def\myAMS{65N30, 65N50}

\def\myAbstract{
For the second lowest order Raviart--Thomas mixed method, we prove that the canonical interpolant and finite element solution for the vector variable in elliptic problems are superclose in the $H(\divg)$-norm on mildly structured meshes, where most pairs of adjacent triangles form approximate parallelograms. We then develop a family of postprocessing operators for Raviart--Thomas mixed elements on triangular grids by using the idea of local least squares fittings. Super-approximation property of the postprocessing operators for the lowest and second lowest order Raviart--Thomas elements is proved under mild conditions. Combining the supercloseness and super-approximation results, we prove that the postprocessed solution superconverges to the exact solution in the $L^{2}$-norm on mildly structured meshes.}

\begin{document}


\ifsiam
  \author{\rebAuthor%
         \thanks{\rebAddress {Email:}\, \rebEmail. \rebThanks}
          \and
          \yulAuthor%
         \thanks{\yulAddress {Email:}\, \yulEmail. \yulThanks}
         }
  \maketitle

  \begin{abstract}\myAbstract\end{abstract}
  \pagestyle{myheadings}
  \thispagestyle{plain}
  \markboth{\rebShortAuthor\ and \yulShortAuthor }{\shortTitle}

\fi   


\ifnummat
   \author{\rebAuthor%
           \and
           \yulAuthor%
           }
  \institute{\rebShortAuthor : \rebAddress\, {Email:}\rebEmail \\
             \yulShortAuthor : \yulAddress\, {Email:}\yulEmail
              }
  \date{Received: April 4, 2019 \  / Accepted: date}
  \maketitle
  \begin{abstract}\myAbstract\end{abstract}
  \begin{subclass}\myAMS \end{subclass}
  \markboth{\rebShortAuthor,\ \yulShortAuthor }{\shortTitle}
\fi


\ifcvs
   \author{\rebAuthor%
         \thanks{\rebShortAuthor : \rebThanks}
         \and
         \yulAuthor%
         \thanks{\yulShortAuthor : \yulThanks}
          }
  \institute{\rebShortAuthor : \rebAddress \, {Email:}\rebEmail \\ 
        \yulShortAuthor  : \yulAddress \, {Email:}\yulEmail }
  \date{Received: \today\  / Accepted: date}
  \maketitle
  \begin{abstract}\myAbstract\end{abstract}
  \begin{subclass}\myAMS \end{subclass}
\fi



\ifmoc


    \author[\yulShortAuthor]{\yulAuthor}
    \address{\yulAddress}
    \email{\yulEmail}

    \subjclass[2010]{Primary \myAMS}
    \date{\today}
    \begin{abstract}\myAbstract\end{abstract}
    \title[Superconvergence of HHJ and Morley elements]{Superconvergence of Hellan--Herrmann--Johnson and Morley elements on quasi-uniform and graded meshes}
    \maketitle
\fi


\ifanm
  \begin{frontmatter}

  \author{\rebAuthor\fnref{fn1}}
  \address{\rebAddress\, {Email:}\rebEmail}
  \fntext[fn1]{\rebThanks} 

  \author{\yulAuthor\corref{cor1}\fnref{nf2}}
  \address{\yulAddress\, {Email:}\yulEmail} 
   \cortext[cor1]{Corresponding Author}
  \fntext[fn2]{\yulThanks} 

  \begin{abstract}\myAbstract\end{abstract}

  \end{frontmatter}

\fi


\ifjcam
  \begin{frontmatter}

  \author{\rebAuthor\fnref{fn1}}
  \address{\rebAddress\, {Email:}\rebEmail}
  \fntext[fn1]{\rebThanks} 

  \author{\yulAuthor\corref{cor1}\fnref{nf2}}
  \address{\yulAddress\, {Email:}\yulEmail} 
   \cortext[cor1]{Corresponding Author}
  \fntext[fn2]{\yulThanks} 

  \begin{abstract}\myAbstract\end{abstract}

  \end{frontmatter}

\fi


\ifjcm
 \markboth{\rebShortAuthor and \yulShortAuthor}
          {\shortTitle}

\author{\rebAuthor\footnote{\rebThanks}\thanks{\rebAddress \\ Email: \rebEmail}
\and
        \yulAuthor\footnote{\yulThanks}\thanks{\yulAddress \\ Email: \yulEmail}}

\maketitle

  \begin{abstract}\myAbstract\end{abstract}
\fi



\section{Introduction and preliminaries}\label{sec1}
Gradient recovery methods for Lagrange elements have been studied extensively by many authors, see, e.g., \cite{ZZ1992,ZZ1992b,BX2003a,BX2003b,BaXuZheng2007,XZ2003,ZhangNaga2005,WuZhang2007} and references therein. Let $u$ be the exact solution of Poisson's equation and $u_{h}$ be the finite element solution from Lagrange elements. In general $\nabla u_{h}$ rather than $u_{h}$ is the main quantity of interest. Gradient recovery methods aim to get a new approximation $\bm{p}_h$ to $\nabla u$ by postprocessing $u_{h}$ or $\nabla u_{h}$. Comparing to $\nabla u_{h}$, $\bm{p}_h$ is often $H^{1}$-conforming and $\bm{p}_h$ superconverges to $\nabla u$ in some situation. In addition, $\bm{p}_h$ can be used to develop a posteriori error estimators. The recovery-based a posteriori error estimators are popular for their simplicity and asymptotic exactness, see, e.g., \cite{ZZ1992b,BX2003b,XZ2003}.

To derive recovery-type superconvergence, a common ingredient is the so-called supercloseness estimate showing that the canonical interpolant and finite element solution are superclose in some norm. In this paper, we consider the second lowest order Raviart--Thomas (denoted by $RT_1$) mixed method for the second order elliptic equation, namely, \eqref{RTmix:dv} with $r=1$. We shall prove that the canonical interpolant $\Pi_h^1\bm{p}$ and the finite element solution $\bm{p}_h^1$ are superclose in the $H(\divg)$-norm under mildly structured grids, i.e., most pairs of adjacent triangles in grids form $O(h^{1+\alpha})$-approximate parallelograms except for a region with measure $O(h^{\beta})$, see Definitions \ref{approxpara1} and \ref{alphasigma}. The supercloseness result in this paper generalizes a result for the $RT_0$ mixed method in \cite{YL2018}. For Poisson's equation, Brandts \cite{Brandts2000} proved a supercloseness estimate for $RT_1$ on three-line grids, i.e., each edge in grids is parallel to one of three fixed lines. 

To relax the restriction on mesh structures in supercloseness analysis, we give a constructive proof for Theorem \ref{RTsuperclosep} instead of using the odd-even argument and the Bramble--Hilbert lemma employed in \cite{Brandts1994,Brandts2000}. For Lagrange elements over $(\alpha,\beta)$-grids, the authors in \cite{BX2003a} transferred the local error $\int_T\nabla(u-u_I)\cdot\nabla v_{h}$ on each element $T$ to line integrals using the divergence theorem, where $u_I$ is the linear Lagrange interpolant. Then line integrals are grouped in terms of tangential components of $\nabla v_{h}$ by delicate triangular integral identities. However, it's not clear how to handle the local error 
$\int_T(\bm{p}-\Pi_h^r\bm{p})\cdot\bm{q}_{h}$ for the $RT_r$ element in a similar fashion. Our key observation is that $RT_r$ elements satisfy the divergence-free property, i.e., $\divg(\bm{p}_{r+1}-\Pi_h^r\bm{p}_{r+1})=0$ on each triangle $T$ provided $\bm{p}_{r+1}\in\mathcal{P}_{r+1}(T)^2$. Hence $\bm{p}_{r+1}-\Pi_h^r\bm{p}_{r+1}=\nabla^\perp w_{r+2}$ for some  $w_{r+2}\in\mathcal{P}_{r+2}(T)$ and it can be handled by Green's theorem, see Section \ref{Proof}. 

For mixed methods, the finite element solution $\bm{p}_{h}$ approximating the vector variable $\bm{p}\in H(\divg,\Omega)$ is the main quantity of physical interest. As far as we know, existing postprocessing/recovery techniques for $\bm{p}$ and $\bm{p}_h$ are restricted to strongly structured grids, e.g., three-line, translation invariant and rectangular grids, see, e.g., \cite{DM1985,Duran1990,DK1998,Brandts2000}. As grids become increasingly unstructured, the rate of superconvergence of $\|\bm{p}-K_h\Pi_h\bm{p}\|_{0,\Omega}$ deteriorates, where $\Pi_h$ is the canonical interpolation and $K_h$ is some postprocessing operator. In addition, most of the existing results of recovery methods focus on the lowest order case while the analysis of recovery operators for higher order elements is limited, especially on irregular grids. In this paper, we construct a new family of recovery operators $R_h^r$ for $RT_r$ ($r\geq0$) elements by fitting the numerical solution $\bm{p}_{h}$ with a vector polynomial of degree $r+1$ in the least squares(LS) sense on each local patch surrounding each vertex in triangular grids. We shall show that $R_h^0$ and $R_h^1$ have nice super-approximation property under mild and easy-to-check conditions. The order of approximation of $R_h^r$ is almost independent of the mesh structure. Combining the supercloseness and $R_h^r$, we finally obtain the superconvergence of the postprocessed $RT_{0}$ and $RT_{1}$ solutions to the exact solution, see Theorem \ref{RTsuper}.

Recovery by local least squares fitting is not a new idea. The famous Zienkiewicz--Zhu(ZZ) superconvergent patch recovery $G_h$ is based on it, see, e.g., \cite{ZZ1992,ZZ1992b}. For linear elements, $\|\nabla u-G_h\nabla u\|=O(h^{2})$ under strongly regular grids (see \cite{LiZhang1999}), that is, each pair of adjacent triangles form an $O(h^{2})$ approximate parallelogram. Alternatively, Zhang and Naga \cite{ZhangNaga2005} proposed a different LS-based patch recovery operator $G_h^r$ for Lagrange elements of degree $r$ by postprocessing the scalar function $u$ rather than $\nabla u$. Roughly speaking, $\|\nabla u-G_h^ru\|=O(h^{r+1})$ provided each LS problem has 
a unique solution on each local patch. $R_h^r$ can be viewed as a Raviart--Thomas version of $G_h^{r+1}$. In practice, the excellent superconvergence property of $G_h^r$ is attributed to the unique solvability of vertex-based LS problems, which is difficult to prove on unstructured grids. For example, \cite{NZ2004} is mainly devoted to the analysis of the uniqueness of the LS solution for $G_h^1$ on unstructured grids. As far as we know, there is no similar analysis for $G_h^r$ with $ r\geq2$. We shall give a practical criterion of uniqueness for $G_h^2$ on unstructured grids, which also works for $R_h^1$, see Theorem \ref{RTuniqueness}.

In this paper, we consider the second order elliptic equation
\begin{subequations}\label{RTmix:c1}
\begin{align}
-\divg(a_2(\bm{x})\nabla u+\bm{a}_1(\bm{x})u)+a_0(\bm{x})u&=f(\bm{x}),\quad \bm{x}\in\Omega,\\
u&=g(\bm{x}),\quad \bm{x}\in\partial\Omega,
\end{align}
\end{subequations}
where $\divg=\nabla\cdot$ is the divergence operator, $a_2, a_0$ are scalar-valued and $\bm{a}_1$ is vector-valued, $\Omega\subset\mathbb{R}^2$ is a bounded and simply-connected Lipschitz domain. Assume that $a_2, \bm{a}_1, a_0$ are sufficiently smooth on $\overline{\Omega}$ and $a_2\geq\Lambda>0$ for some constant $\Lambda.$
Let
\begin{equation*}
\begin{aligned}
\bm{p}&=a_2\nabla u+\bm{a}_1u,\\
a&=a_2^{-1},\quad\bm{b}=a_2^{-1}\bm{a}_1,\quad c=a_0.
\end{aligned}
\end{equation*}
Equation
\eqref{RTmix:c1} is equivalent to the first order system
\begin{subequations}\label{RTmix:c}
\begin{align}
a\bm{p}-\bm{b}u-\nabla u&=0,\quad\bm{x}\in\Omega,\\
-\divg\bm{p}+cu&=f,\quad\bm{x}\in\Omega,\\
u&=g,\quad\bm{x}\in\partial\Omega.
\end{align}
\end{subequations}
Let 
$\CQ=H(\divg,\Omega):=\{\bm{q}\in L^{2}(\Omega)^{2}:
\divg\bm{q}\in L^{2}(\Omega)\}$ and $\mathcal{V}=L^{2}(\Omega).$
The mixed formulation for \eqref{RTmix:c} is to find the pair $\{\bm{p},u\}\in\mathcal{Q}\times\mathcal{V}$, such that
\begin{subequations}\label{RTmix:v}
\begin{align}
(a\bm{p},\bm{q})-(\bm{q},\bm{b}u)+(\divg\bm{q},u)&=\langle\bm{q}\cdot\bm{n},g\rangle,\label{RTmix:va}\\
-(\divg\bm{p},v)+(cu,v)&=(f,v),
\end{align}
\end{subequations}
for each pair $\{\bm{q},v\}\in\mathcal{Q}\times\mathcal{V}$. Here $\langle\cdot,\cdot\rangle$ denotes the $L^2$-inner product on $\partial\Omega.$

Let $\Th$ be a collection of triangles that forms a triangulation of $\Omega.$ Let $h_T=|T|^\frac{1}{2}$ be the diameter of $T$, where $|T|$ is the area of $T$. Let $h=\max_{T\in\Th} h_T<1$ be the mesh-size. $\Th$ is assumed to quasi-uniform, namely, $\max_{T\in\Th}h_T\leq C_0(\min_{T\in\Th}h_T)$ for some generic constant $C_0$. The quasi-uniformity implies the minimum angle condition (MAC), namely, there exists a fixed constant $\Theta>0$, such that $\theta\geq\Theta>0$ for any angle $\theta$ of any triangle $T\in\Th$. Given a one-dimensional or two-dimensional subset $U\subset\mathbb{R}^{2}$, let
$$\CP_r(U)=\{v: v\ \text{is\ a\ polynomial\ on } U \text{ of\ degree }\leq r\}$$
denote the space of  polynomials of degree $\leq r.$ Let $\Eh, \Eh^{o}, \Eh^{\partial}$ denote the set of edges, interior edges and boundary edges in $\Th$ , respectively. Let $\Nh$ denote the set of vertices in $\Th$. Several kinds of local patches are useful for finite element superconvergence analysis. For $\bm{z}\in\CN_h,$ let $\omega_z$ be the union of triangles in $\Th$ sharing $z$ as a vertex. For $e\in\CE_h,$ let $\omega_e$ be the union of triangles in $\Th$ sharing $e$ as an edge. For $T\in\Th,$ let $\omega_T$ be the union of $T$ and triangles in $\Th$ sharing at least one vertex with $T$. The local nodes, edges , and triangles in $U$ are $\Nh(U)=\{z\in\Nh: z\in\bar{U}\}$, $\Eh(U)=\{e\in\Eh: e\subset\bar{U}\}$, and $\Th(U)=\{T\in\Th: T\subset\bar{U}\}$, respectively. 

For $r\geq0$ and $T\in\Th$, define the space of shape functions
\begin{equation}\label{formRT}
\mathcal{RT}_r(T):=\left\{\begin{pmatrix}v_1\\v_2\end{pmatrix}+v_3\begin{pmatrix}x_1\\x_2\end{pmatrix}: v_i\in\CP_r(T), i=1,2,3\right\}.
\end{equation}
The $RT_r$ finite element spaces are
\begin{displaymath}
\begin{aligned}
&\mathcal{Q}_{h}^r:=\left\{\bm{q}_{h}\in\mathcal{Q}: \bm{q}_{h}|_T\in\mathcal{RT}_r(T),\ \forall T\in\mathcal{T}_{h}\right\},\\
&\mathcal{V}_{h}^r:=\{v_{h}\in\mathcal{V}: v_{h}|_T\in\mathcal{P}_r(T),\ \forall T\in\mathcal{T}_{h}\}.
\end{aligned}
\end{displaymath} 
The mixed method for \eqref{RTmix:v} is to find $\{\bm{p}_{h}^r,u_{h}^r\}\in\mathcal{Q}_{h}^r
\times\mathcal{V}_{h}^r$, such that
\begin{subequations}\label{RTmix:dv}
\begin{align}
(a\bm{p}_{h}^r,\bm{q}_{h})-(\bm{q}_{h},\bm{b}u_{h}^r)
+(\divg\bm{q}_{h},u_{h}^r)&=\langle\bm{q}_{h}\cdot\bm{n},g\rangle,&&\quad\bm{q}_{h}\in\mathcal{Q}_{h}^r,\label{RTmix:dva}\\
-(\divg\bm{p}_{h}^r,v_{h})+(cu_{h},v_{h})&=(f,v_{h}),&&\quad v_{h}\in\mathcal{V}_{h}^r.
\end{align}
\end{subequations}
Under mild assumptions, Douglas and Roberts \cite{DR1985} proved the well-posedness and a priori error estimates for the method \eqref{RTmix:dv}.

Given a positive integer $s$ and a sufficiently smooth function $v$, let 
$$|D^sv|:=\sum_{\alpha_1+\alpha_2=s}\left|\frac{\partial^{\alpha_1+\alpha_2}}{\partial x_1^{\alpha_1}\partial x_2^{\alpha_2}}v\right|.$$
For a domain $U$, the Sobolev seminorms and norms are defined by
\begin{equation*}
\begin{aligned}
&|v|_{s,p,U}=\big(\int_U|D^sv|^{p}\big)^{\frac{1}{p}},\quad
\|v\|_{s,p,U}=\big(\sum_{m=0}^s|v|_{m,p,U}^{p}\big)^{\frac{1}{p}},\\
&|v|_{m,U}=|v|_{m,2,U},\quad\|v\|_{m,U}=\|v\|_{m,2,U},
\end{aligned}
\end{equation*}
Sobolev norms with $\infty$-index and norms of vector-valued functions are generalized in usual ways. 

Let $|v|_{h,m,U}:=\big(\sum_{T\in\Th}|v|_{m,T}^2\big)^\frac{1}{2}$ denote the mesh-dependent semi-norm w.r.t.~$\Th$. We say $A\lesssim B$ provided $A\leq CB,$ where $C$ is a generic constant that may change from line to line, and depends only on the shape regularity of $\Th$ measured by $C_0$ or $\Theta.$ We say $A\approx B$ if $A\lesssim B$ and $B\lesssim A.$ The regularity condition will be indicated on right hand sides of estimates. {In addition to $\CQ_h^r$ and $\CV_h^r$, we need the standard nodal finite element space
\begin{align*}
    \mathcal{W}_h^r=\{w\in C(\Omega): w|_T\in\mathcal{P}_r(T),\ \forall T\in\Th\},
\end{align*}
where $C(\Omega)$ is the space of continuous functions on $\Omega$.} We present two well-known inequalities that will be used in the rest of this paper. 
\begin{theorem}[Interpolation error]
{Let $I_h^r: C(\Omega)\rightarrow\mathcal{W}_h^r$ denote the Lagrange interpolation of degree $r$.} For $T\in\Th$ and $r\geq1,$ it holds that
\begin{equation}\label{feinterp}
\|v-I_h^rv\|_{0,\gamma,T}\lesssim 
h^{r+\frac{2}{\gamma}}|v|_{h,r+1,T},\quad1\leq \gamma\leq\infty.
\end{equation} 
\end{theorem}
\begin{theorem}[Trace inequalities ]
For $T\in\Th$ and $v\in H^1(T)$, it holds that
\begin{equation}\label{trace}
\begin{aligned}
    \|v\|_{0,\partial T}\lesssim h_T^{-\frac{1}{2}}\|v\|_{0,T}+h_T^{\frac{1}{2}}\|\nabla v\|_{0,T}.
\end{aligned}
\end{equation}
\end{theorem}

\section{Local error expansions}\label{LEE}
\begin{figure}[tbhp]
\centering
\includegraphics[width=13.0cm,height=4.0cm]{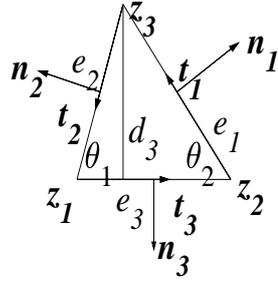}
\caption{A local triangle $T$ and associated quantities.}
\label{triangle}
\end{figure}

We begin with geometric identities on a local element $T$. It has three vertices $\{z_{k}\}_{k=1}^{3}$, oriented counterclockwise,
and corresponding barycentric coordinates $\{\lambda_{k}\}_{k=1}^{3}$.
Let $e_{k}$ denote the edge opposite to $z_{k}$, $\theta_{k}$ the angle opposite to $e_{k}$, $\ell_{k}$ the length of $e_{k}$, $d_{k}$ the distance from $z_{k}$ to $e_{k}$,
$\bm{t}_{k}$ the unit tangent to $e_{k}$, oriented counterclockwise, $\bm{n}_{k}$ the unit outward normal
to $e_{k}$, $\dtk$ the tangential derivative, $\dnk$ the normal derivative, and $\dtnk$ the second mixed derivative, see Fig.~\ref{triangle}. Corresponding quantities on triangles $T^{\prime}$ and $T^{\prime\prime}$ have superscripts $\prime$ and $\prime\prime$ respectively. The subscripts are equivalent mod 3, e.g., $\ell_4=\ell_1, \theta_0=\theta_3.$

We have the rotational gradient 
$\nabla^\perp v=\left(-\partial_{x_2} v,\partial_{x_1}v\right)^\intercal$,
and the adjoint 
$\nabla\times\bm{q}=\partial_{x_1}q_2-\partial_{x_2} q_1.$
$\nabla^\perp$ and $\nabla\times$ are related by Green's formula
\begin{equation}\label{IP}
\int_T\nabla^\perp w\cdot\bm{q}=\int_{\partial T}w\bm{q}\cdot\bm{t}-\int_Tw\nabla\times\bm{q},
\end{equation}
where $\bm{t}$ is the unit tangent to $\partial K$ oriented counterclockwise. For $\bm{v}\in\mathbb{R}^2$, define $\bm{v}^\perp=(-v_2,v_1)$. Clearly, $\bm{n}_{k}^\perp=\bm{t}_{k},$ $\bm{t}_{k}^\perp=-\bm{n}_{k}.$ 

Now we introduce basic definitions for $RT_r$ elements. For $e\in\Eh,$ let $\{(w_j,\bm{g}_{j})\}_{j=1}^{r+1}$ denote the Gaussian quadrature rule on $e$, where $\{\bm{g}_j\}$ are quadrature points and $\{w_{j}\}$ are corresponding weights. $\{(w_j,\bm{g}_{j})\}_{j=1}^{r+1}$ is exact for $\mathcal{P}_{2r+1}(e)$, i.e.,
\begin{equation}\label{quad}
\int_ev=|e|\sum_{j=1}^{r+1}w_jv(\bm{g}_j)\text{ for all }v\in\mathcal{P}_{2r+1}(e),   
\end{equation}
where $|e|$ is the length of $e.$
Let
$v_{j}\in\mathcal{P}_r(e)$ be the polynomial that is $w_{j}^{-1}$ at $\bm{g}_{j}$ and $0$ at the rest of quadrature points. For $T\in\Th,$ let $\{\lambda_l\}_{l=1}^{r(r+1)/2}$ be the nodal basis function of Lagrange elements of degree $r-1$ on $T$ ($\{\lambda_l\}=\emptyset$ if $r=0$; $\{\lambda_l\}=\{1\}$ if $r=1$). We can specify degrees of freedom of $RT_r$ elements as 
\begin{equation*}
\mathcal{N}_{e}^{j}(\bm{q}):=\frac{1}{|e|}\int_{e}\bm{q}\cdot\bm{n}_{e}v_{j},\quad
\mathcal{N}_T^{lm}(\bm{q}):=\frac{1}{|T|}\int_Tq_m\lambda_l,
\end{equation*}
where $\bm{n}_e$ is a unit normal to $e$, $\bm{q}=(q_1,q_2)^\intercal$, and $1\leq j\leq r+1$, $1\leq l\leq r(r+1)/2, m=1,2.$ {\color{blue}By \eqref{quad} and the definition of $v_j$, we have $\mathcal{N}_{e}^j(\bm{q})=\bm{q}(\bm{g}_j)\cdot\bm{n}_{e}$ provided $\bm{q}\in\mathcal{P}_{r+1}(e)^2$.}
For $\bm{q}\in H^1(\Omega)^2$, the $RT_r$ interpolant $\Pi_{h}^r\bm{q}\in\CQ_h^r$ satisfies 
\begin{equation*}
\begin{aligned}
\mathcal{N}_e^j(\Pi_{h}^r\bm{q})&=\mathcal{N}_e^j(\bm{q}),\quad\mathcal{N}_T^{lm}(\Pi_{h}^r\bm{q})=\mathcal{N}_T^{lm}(\bm{q}),
\end{aligned}
\end{equation*}
for all indices $j,l,m$, and $e\in\Eh, T\in\Th$. 
The existence and uniqueness of $\Pi_h^r\bm{q}$ is always guaranteed. In addition, $\Pi_h^r$ is stable in the $L^\infty$-norm
\begin{equation}\label{Linfty}
\|\Pi_h^k\bm{q}\|_{0,\infty,T}\lesssim\|\bm{q}\|_{0,\infty,T},\quad T\in\Th.
\end{equation}
For $v\in\mathcal{V}$, the interpolant $P_{h}^rv$ is the $L^{2}$-projection of $v$ onto $\mathcal{V}_{h}^r$.
There is a nice commuting property about $P^r_{h}$, $\Pi^r_{h}$ and $\divg$, i.e.,
\begin{equation}\label{commutative}
\divg(\Pi_h^r\bm{q})=P_h^r(\divg\bm{q}),\quad\forall\bm{q}\in H^1(\Omega)^2.
\end{equation}
The following interpolation error estimates hold, see, e.g., \cite{DR1985}.
\begin{subequations}\label{approxRT}
\begin{align}
&\|\bm{q}-\Pi_{h}^r\bm{q}\|_{0,\Omega}\lesssim h^{r+1}|\bm{q}|_{h,r+1,\Omega},\label{approxRTa}\\
&\|\divg(\bm{q}-\Pi_{h}^r\bm{q})\|_{0,\Omega}\lesssim h^{r+1}|\divg\bm{q}|_{h,r+1,\Omega},\\
&\|v-P_{h}^rv\|_{0,\Omega}\lesssim h^{r+1}|v|_{h,r+1,\Omega}.
\end{align}
\end{subequations}

In the rest of this section, we will present variational error expansions for the $RT_1$ element. Comparing to $RT_0$, the theory of $RT_1$ is much more complicated. Let $d$ be the diameter of the circumscribed circle of $T$. For each edge $e_k$, there are several associated geometric quantities 
\begin{equation*}
\begin{aligned}
&\mu^1_{11,k}
=\frac{1}{5760}\big(3\ell_{k}^{4}-3(\ell_{k-1}^{2}-\ell_{k+1}^{2})^{2}
-4\ell_{k}^{2}(\ell_{k-1}^{2}+\ell_{k+1}^{2})\big),\\
&\mu^1_{12,k}=\mu^1_{21,k}=\frac{1}{1440d}
\ell_{1}\ell_{2}\ell_{3}(\ell_{k-1}^{2}-\ell_{k+1}^{2}),
\quad\mu^1_{22,k}=-\frac{1}{1440d^{2}}
\ell_{1}^{2}\ell_{2}^{2}\ell_{3}^{2},\\
&\mu^2_{11,k}
=\frac{1}{2880\ell_{1}\ell_{2}\ell_{3}}d(\ell_{k-1}^{2}-\ell_{k+1}^{2})\big(4\ell_k^2-(\ell_{k-1}^{2}-\ell_{k+1}^{2})^{2}
-3\ell_{k}^{2}(\ell_{k-1}^{2}+\ell_{k+1}^{2})\big),\\
&\mu^2_{12,k}=\mu^2_{21,k}
=-\mu^1_{11,k},\quad
\mu^2_{22,k}=-\mu^1_{12,k},
\end{aligned}
\end{equation*}
and second order differential operators $\{\DD_{i,k}^{jl}\}_{1\leq i,j,l\leq2}$
\begin{equation*}
    \begin{aligned}
    \DD_{1,k}^{11}&=\bm{t}_k\cdot\dtk^2,\quad\DD_{1,k}^{12}=\DD_{1,k}^{21}=\bm{t}_k\cdot\dtnk,\quad\DD_{1,k}^{22}=\bm{t}_k\cdot\dnk^2,\\
    \DD_{2,k}^{11}&=\bm{n}_k\cdot\dtk^2,\quad\DD_{2,k}^{12}=\DD_{2,k}^{21}=\bm{n}_k\cdot\dtnk,\quad\DD_{2,k}^{22}=\bm{n}_k\cdot\dnk^2.
    \end{aligned}
\end{equation*}
We define the second order differential operator $\CB_k(\bm{q}):=\sum_{i,j,l=1}^2\mu_{jl,k}^i\DD_{i,k}^{jl}(\bm{q}).$ The next lemma is our main tool for estimating the global variational error whose proof is left in Section \ref{Proof}.
\begin{lemma}\label{RT1err2}
For $\bm{p}_2\in\CP_2(T)^2$ and $w_{2}\in\mathcal{P}_{2}(T)$,
\begin{equation*}
\begin{aligned}
\int_T(\bm{p}_{2}-\Pi_h^1\bm{p}_{2})\cdot\nabla^\perp w_{2}=\sum_{k=1}^{3}\int_{e_k}\CB_k(\bm{p}_2)\dtk^2{w_{2}}.
\end{aligned}
\end{equation*}
\end{lemma}
Built upon Lemma \ref{RT1err2}, we derive the local error expansion for general $\bm{p}$.
\begin{theorem}\label{RT1err3}
For $w_2\in\mathcal{P}_2(T)$,
\begin{equation*}
\begin{aligned}
\int_T(\bm{p}-\Pi_h^1\bm{p})\cdot\nabla^\perp w_2&=\sum_{k=1}^{3}
\int_{e_{k}}\CB_k(\bm{p})\dtk^2{w_2}+O(h_T^3)|\bm{p}|_{3,T}\|\nabla^\perp w_2\|_{0,T}.
\end{aligned}
\end{equation*}
\end{theorem}
\begin{proof}
Let $\bm{p}_I$ be the quadratic interpolant of $\bm{p}$. By Lemma \ref{RT1err2}, we have
\begin{equation}\label{RT1totallocal}
\begin{aligned}
&\int_T(\bm{p}-\Pi_h^1\bm{p})\cdot\nabla^\perp w_2=\int_T(\text{id}-\Pi_h^1)(\bm{p}-\bm{p}_I)\cdot\nabla^\perp w_2\\
&\qquad+\sum_{k=1}^{3}\int_{e_{k}}\CB_k{(\bm{p}_I-\bm{p})}\dtk^2{w_2}+\sum_{k=1}^{3}\int_{e_{k}}\CB_k({\bm{p}})\dtk^2{w_2}\\
&\quad:=I+II+III,\\
\end{aligned}
\end{equation}
{\color{blue}where id is the identity operator.} The inequalities \eqref{feinterp} and \eqref{Linfty} give the upper bound 
\begin{equation}\label{RT1bd12}
\begin{aligned}
|I|&\lesssim\|(\text{id}-\Pi_h^1)(\bm{p}-\bm{p}_I)\|_{0,T}\|\nabla^\perp w_2\|_{0,T}\\
&\lesssim h_T\|(\text{id}-\Pi_h^1)(\bm{p}-\bm{p}_I)\|_{0,\infty,T}\|\nabla^\perp w_2\|_{0,T}\\
&\lesssim h_T\|\bm{p}-\bm{p}_I\|_{0,\infty,T}\|\nabla^\perp w_2\|_{0,T}\\
&\lesssim h_T^3|\bm{p}|_{3,T}\|\nabla^\perp w_2\|_{0,T}.
\end{aligned}
\end{equation}
Using the trace inequality \eqref{trace}, inverse inequality, and $\mu^i_{jl,k}=O(h_T^4),$
\begin{equation}\label{RT1bd3}
\begin{aligned}
|II|&\lesssim\sum_{k=1}^3\|\CB_k{(\bm{p}_I-\bm{p})}\|_{0,e_k}\|\dtk^2{w_2}\|_{0,e_k}\\
&\lesssim \sum_{k=1}^3\big(h_T^{-\frac{1}{2}}\|\CB_k{(\bm{p}_I-\bm{p})}\|_{0,T}+h_T^{\frac{1}{2}}|\CB_k{(\bm{p}_I-\bm{p})}|_{1,T}\big)\\
&\qquad\times \big(h_T^{-\frac{1}{2}}\|D^2w_2\|_{0,T}+h_T^{\frac{1}{2}}|D^2w_2|_{1,T}\big)\\
&\lesssim \sum_{k=1}^3(h_T^{-\frac{1}{2}}|h_T^4(\bm{p}_I-\bm{p})|_{2,T}+h_T^{\frac{1}{2}}|h_T^4(\bm{p}_I-\bm{p})|_{3,T})\times(h^{-\frac{3}{2}}\|\nabla^\perp w_2\|_{0,T})\\
&\lesssim h_T^3|\bm{p}|_{3,T}\|\nabla^\perp w_2\|_{0,T}.
\end{aligned}
\end{equation}
Combining \eqref{RT1totallocal}--\eqref{RT1bd3} , we prove the theorem.
\qed\end{proof}
Our supercloseness estimates in this paper hold on mildly structured grids described as follows, see, e.g., \cite{LMW2000,BX2003a,XZ2003,NZ2004,HuangXu2008}.
\begin{definition}\label{approxpara1}
For $e\in\Eh^o$, let $T, T^{\prime}\in\Th$ be the two adjacent
elements sharing $e$. Define $e_1=e_1^\prime=e$. By going along $\partial T$ and $\partial T^{\prime}$ counterclockwise, we obtain other two pairs of corresponding edges $e_2, e_2^\prime$ and $e_3, e_3^\prime$. We say $\omega_e=T\cup T^{\prime}$ is an $O(h^{1+\alpha})$-approximate parallelogram provided $ |e_i|=|e_i^\prime|+O(h^{1+\alpha})$ for $i=1,2,3$.
\end{definition}
\begin{definition}\label{alphasigma}
{\color{blue}Assume $\mathcal{E}_{h}^{o}$ is the disjoint union of two subsets $\mathcal{E}_{h,1}^{o}$ and $\mathcal{E}_{h,2}^{o}$}.
We say the triangulation $\mathcal{T}_{h}$ satisfies the $(\alpha,\beta)$-condition provided for each $e\in\mathcal{E}_{h,1}^{o}$, $\omega_e$ an $O(h^{1+\alpha})$-approximate parallelogram, while $\sum_{e\in\mathcal{E}_{h,2}^{o}}|\omega_e|=O(h^{\beta})$. 
\end{definition}
Although the expression of $\CB_k$ is complicated, it suffices to keep the following in mind.
\begin{enumerate}
\item $\{\CB_k\}_{k=1}^3$ are second order differential operators of magnitude $h_T^4$:
$$\CB_k(\bm{q})=O(h_T^4)\sum_{i,j,l=1}^2\partial_{x_i}\partial_{x_j}q_l.$$
\item
For $e\in\Eh^o$, we have $\omega_e=T\cup T^\prime$. Let $\bm{t}_{e}$ denote the unit tangent and and $\bm{n}_{e}$ the unit normal to $e$ whose directions are induced by $T$. Let $\bar{{a}}=\frac{1}{|T|}\int_T{a}$ and $\bar{a}^\prime=\frac{1}{|T|}\int_{T^\prime}{a}.$ 
Let $\CB_e$ be the operator based on $T$ and $\CB_e^\prime$ based on $T^\prime$. If $\omega_e$ is an $O(h^{1+\alpha})$-approximate parallelogram, then on the edge $e$, we have the cancellation 
\begin{equation}\label{cancelRT1}
\bar{{a}}\CB_e(\bm{q})-\bar{{a}}^\prime\CB_e^\prime(\bm{q})=O(h_e^{4+\min(1,\alpha)})\sum_{i,j,m=1}^2\partial_{x_i}\partial_{x_j}q_m.
\end{equation}
\end{enumerate}
Indeed, $\omega_e$ is an approximate parallelogram implies that $\ell_k=\ell_k^\prime+O(h^{1+\alpha})$, $\bm{t}_k=\bm{t}_k^\prime+O(h^{\alpha})$, $\sin\theta_k=\sin\theta_k^\prime+O(h^\alpha)$, $d=d^\prime+O(h^{1+\alpha})$.
Combining these estimates with $\bar{{a}}=\bar{{a}}^\prime+O(h)$, \eqref{cancelRT1} follows from the telescoping type inequality $$\left|\prod_{i=1}^{n}a_{i}-\prod_{i=1}^{n}b_{i}\right|\leq \sum_{i=1}^{n}|a_{i}-b_{i}|\prod_{j\neq i}\max(a_{j},b_{j}).$$

\section{Supercloseness estimates}
In this section, first we prove a superconvergence estimate for variational error which is a foundation of supercloseness estimates. 
\begin{lemma}\label{RT1mainlemma}
Let $\mathcal{T}_{h}$ satisfy the $(\alpha,\beta)$-condition and $\bar{a}$ be the piecewise constant with $\bar{a}|_T=\frac{1}{|T|}\int_Ta$ for each $T\in\Th$.
For $w_h\in\mathcal{W}_{h}^2$, it holds that
\begin{equation*}
(\bar{a}(\bm{p}-\Pi_{h}^1\bm{p}),\nabla^\perp w_h)
\lesssim h^{2+\min(\frac{1}{2},\alpha,\frac{\beta}{2})}\big(|\bm{p}|_{2,\infty,\Omega}+|\bm{p}|_{3,\Omega}\big)\|\nabla^\perp w_h\|_{0,\Omega}.
\end{equation*}
\end{lemma}
\begin{proof}
By Theorem \ref{RT1err3} and the Cauchy--Schwarz inequality, the left hand side is
\begin{equation}\label{RT1maintotal}
\begin{aligned}
&(\bar{{a}}(\bm{p}-\Pi_{h}^1\bm{p}),\nabla^\perp w_h)\\
&\quad=\sum_{T\in\Th}\sum_{k=1}^3\int_{e_k}\bar{a}\CB_k(\bm{p})\dtk^2{w_h}+\sum_{T\in\Th}O(h_T^3)|\bm{p}|_{3,T}\|\nabla^\perp w_h\|_{0,T}\\
&\quad=\big(\sum_{e\in\CE_{h,1}^o}+\sum_{e\in\CE_{h,2}^o\cup\Eh^\partial}\big)\int_e\big(\bar{a}\CB_e(\bm{p})-\bar{a}^\prime\CB_e^\prime(\bm{p})\big)\dte^2{w_h}\\
&\quad\qquad+O(h^3)|\bm{p}|_{3,\Omega}\|\nabla^\perp w_h\|_{0,\Omega}:=I+II+O(h^3)|\bm{p}|_{3,\Omega}\|\nabla^\perp w_h\|_{0,\Omega}.
\end{aligned}
\end{equation}
Here the notations in \eqref{cancelRT1} are adopted and $\CB_e^\prime(\bm{p})=0$ if $e\in\Eh^\partial.$ By the cancellation \eqref{cancelRT1}, the trace inequality \eqref{trace}, and the inverse inequality,
\begin{equation}\label{RT1sumI}
    \begin{aligned}
    |I|&\lesssim\sum_{e\in\CE_{h,1}^o}h^{4+\min(1,\alpha)}\|D^2\bm{p}\|_{0,e}\|D^2w_h\|_{0,e}\\
    &\lesssim\sum_{e\in\CE_{h,1}^o}h^{4+\min(1,\alpha)}\big(h^{-\frac{1}{2}}\|D^2\bm{p}\|_{0,T}+h^{\frac{1}{2}}\|D^3\bm{p}\|_{0,T}\big)\big(h^{-\frac{1}{2}}\|D^2w_h\|_{0,T}\big)\\
    &\lesssim\sum_{e\in\CE_{h,1}^o}h^{2+\min(1,\alpha)}\|\bm{p}\|_{3,T}\|\nabla^\perp w_h\|_{0,T}\\
    &\lesssim h^{2+\min(1,\alpha)}\|\bm{p}\|_{3,\Omega}\|\nabla^\perp w_h\|_{0,\Omega}.
    \end{aligned}
\end{equation}
For $e\in\mathcal{E}_{h,2}^o$, there is no cancellation. Let $\widetilde{\Omega}=\cup_{e\in\CE_{h,2}^o\cup\CE_h^\partial}\omega_e.$ Using $|\widetilde{\Omega}|=O(h^{\min(1,\beta)})$ and the inverse inequality, the sum over $\mathcal{E}_{h,2}^o$ is 
\begin{equation}\label{RT1sumII}
\begin{aligned}
|II|&\lesssim\sum_{e\in\CE_{h,2}^o\cup\CE_h^\partial}h^4|D^2\bm{p}|_{0,\infty,e}\int_e|\dtk^2w_h|\\
&\lesssim h^2|\bm{p}|_{2,\infty,\Omega}\sum_{e\in\CE_{h,2}^o\cup\CE_h^\partial}\int_{\omega_e}|\nabla^\perp w_h|\\
&\lesssim h^{2+\min(\frac{1}{2},\frac{\beta}{2})}|\bm{p}|_{2,\infty,\Omega}\|\nabla^\perp w_h\|_{0,\widetilde{\Omega}}.
\end{aligned}
\end{equation}
Combining \eqref{RT1maintotal}--\eqref{RT1sumII} we prove the theorem.
\qed\end{proof}
Subtracting \eqref{RTmix:dv} from \eqref{RTmix:v} gives the error equation
\begin{subequations}\label{RTerreqn}
\begin{align}
(\bm{a}(\bm{p}-\bm{p}^r_{h}),\bm{q}_{h})-(\bm{q}_{h},\bm{b}(u-u^r_{h}))
+(\divg\bm{q}_{h},u-u^r_{h})&=0,\quad\bm{q}_{h} \in\mathcal{Q}_{h}^r,\\
-(\divg(\bm{p}-\bm{p}^r_{h}),v_{h})+(c(u-u^r_{h}),v_{h})&=0,\quad v_{h}\in\mathcal{V}_{h}^r.
\end{align}
\end{subequations}
Douglas and Roberts \cite{DR1985} have shown the standard a priori error estimates:
\begin{equation}\label{apriori}
\begin{aligned}
&\|\bm{p}-\bm{p}^r_{h}\|_{0,\Omega}\lesssim h^{r+1}\|u\|_{r+2,\Omega},\\
&\|\divg(\bm{p}-\bm{p}^r_{h})\|_{0,\Omega}\lesssim h^{r+1}\|u\|_{r+3,\Omega},\\
&\|u-u_{h}^r\|_{0,\Omega}\lesssim h^{r+1}\|u\|_{r+1+\delta_{r0},\Omega},
\end{aligned}
\end{equation}
{\color{blue}where $\delta_{r0}=1$ if $r=0$ and $\delta_{r0}=0$ if $r\neq0$.} In addition,
\cite{DR1985} gives the well-known supercloseness result for the scalar unknown $u$
\begin{equation}\label{supercloseu}
    \|P^r_h u-u^r_h\|_{0,\Omega}\lesssim h^{r+2}\|u\|_{r+2+\delta_{r0},\Omega}.
\end{equation}
\eqref{supercloseu} holds on unstructured meshes and implies that $\|\divg(\Pi_{h}^r\bm{p}-\bm{p}^r_{h})\|_{0,\Omega}$ is supersmall. For convenience, let
$\bm{\xi}_h:=\Pi^r_h\bm{p}-\bm{p}^r_h$.
\begin{theorem}\label{superclosedivp}
For general shape regular $\Th$ and $r\geq0,$
\begin{equation*}
\|\divg(\Pi_{h}^r\bm{p}-\bm{p}_{h}^r)\|_{0,\Omega}\lesssim h^{r+2}\|u\|_{2+r+\delta_{r0},\Omega}.
\end{equation*}
\end{theorem}
\begin{proof}
Let 
$$v_{h}:=\frac{\divg\bm{\xi}_h}{\|\divg\bm{\xi}_h\|_{0,\Omega}}\in\mathcal{V}_h^r.$$
By \eqref{commutative} and \eqref{RTerreqn}, we have
\begin{equation*}
\begin{aligned}
&\|\divg\bm{\xi}_h\|_{0,\Omega}=(\divg\bm{\xi}_h,v_{h})=(P^r_{h}\divg\bm{p}-\divg\bm{p}^r_{h},v_{h})\\
&\quad=(\divg(\bm{p}-\bm{p}^r_{h}),v_{h})=(u-P^r_{h}u,cv_{h})+(P^r_{h}u-u^r_{h},cv_{h}).
\end{aligned}
\end{equation*}
It then follows from \eqref{approxRT}, \eqref{apriori}, \eqref{supercloseu}, and $\|v_{h}\|_{0,\Omega}=1$ that
\begin{equation*}
\begin{aligned}
\|\divg\bm{\xi}_h\|_{0,\Omega}
&=(u-P_h^ru,cv_{h}-P^r_{h}(cv_{h}))+O(h^{r+2})\|u\|_{2+r+\delta_{r0},\Omega}\\
&=O(h^{2r+2})\|u\|_{r+1,\Omega}|cv_{h}|_{h,r+1,\Omega}+O(h^{r+2})\|u\|_{2+r+\delta_{r0},\Omega}\\
&=O(h^{r+2})\|u\|_{2+r+\delta_{r0},\Omega}.
\end{aligned}
\end{equation*}
In the last step, we use $v_h|_T\in\mathcal{P}_r(T)$ and the inverse inequality.
\qed\end{proof}

Before proving the superconvergence estimate of $\|\Pi^r_{h}\bm{p}-\bm{p}^r_{h}\|_{0,\Omega}$, it is necessary to discuss the $L^2$ de Rham complex in $\mathbb{R}^2$:
\begin{equation*}
H^1(\Omega)\xrightarrow{\nabla^\perp}\CQ\xrightarrow{\divg}\CV\rightarrow0.
\end{equation*}
Here $\CV=L^2(\Omega)$ is equipped with the standard $(\cdot,\cdot)$ inner product. Since we are dealing with variable coefficients, $\CQ$ is equipped with the weighted $L^2$ inner product $(\cdot,\cdot)_a$ given by
$$(\bm{q}_1,\bm{q}_2)_a:=(a\bm{q}_1,\bm{q}_2),\quad\bm{q}_1,\bm{q}_2\in L^2(\Omega)^2.$$ 
The weighted $L^2$-norm is $\|\bm{q}\|_a=(a\bm{q},\bm{q})^\frac{1}{2}.$ Clearly, $\|\bm{q}\|_{0,\Omega}\approx\|\bm{q}\|_a$ for all $\bm{q}\in L^2(\Omega)^2.$ 
Similarly, we have the discrete subcomplex
\begin{equation}\label{DRcomplex}
\mathcal{W}_h^{r+1}\xrightarrow{\nabla^\perp}\CQ_h^r\xrightarrow{\divg}\CV_h^r\rightarrow0.
\end{equation}
Let $\oplus$ denote the direct sum w.r.t.~$(\cdot,\cdot)_a.$ Since $\Omega$ is simply connected, \eqref{DRcomplex} is exact and the discrete Helmholtz/Hodge decomposition (see, e.g., \cite{AFW2006,AFW2010,CHX2009,YL2019}) holds:
\begin{equation}\label{disHelmholtz} 
\mathcal{Q}_{h}^r=\nabla^\perp\mathcal{W}_{h}^{r+1}\oplus\grad_{h}\mathcal{V}_{h}^r,
\end{equation}
where $\grad_{h}: \mathcal{V}_{h}^r\to\mathcal{Q}_{h}^r$ is the adjoint of $-\divg: \CQ_h^r\rightarrow \CV_h^r$ w.r.t. the weighted inner product $(\cdot,\cdot)_a$, namely,
$(a\grad_{h}v_{h},\bm{q}_{h})=-(v_{h},\divg\bm{q}_{h})$ for all $\bm{q}_{h}\in\mathcal{Q}_{h}^r.$

The last ingredient for our supercloseness analysis is a discrete Poincar\'e inequality. 
\begin{lemma}\label{disPoincare}
\begin{equation*}
\|v_{h}\|_{0,\Omega}\lesssim\|\grad_{h}v_{h}\|_{a},\quad v_h\in\CV_h^r.
\end{equation*}
\end{lemma}
\begin{proof}
$\divg: \CQ^r_h\rightarrow\CV^r_h$ is surjective and there exists $\bm{q}_{h}\in\mathcal{Q}^r_{h}$ and $\divg\bm{q}_h=v_h.$ In addition, $\bm{q}_h$ can be chosen (see \cite{RT1977}) such that 
$\|\bm{q}_h\|_a\approx\|\bm{q}_{h}\|_{0,\Omega}\lesssim\|v_{h}\|_{0,\Omega}.$
It then follows
\begin{equation*}
\begin{aligned}
\|v_{h}\|_{0,\Omega}^{2}&=-(a\grad_{h}v_{h},\bm{q}_{h})\lesssim\|\grad_{h}v_{h}\|_{a}\|v_{h}\|_{0,\Omega},
\end{aligned}
\end{equation*}
which completes the proof.
\qed\end{proof}
With the above preparations, we are able to prove supercloseness estimates for the $RT_1$ mixed methods.
\begin{theorem}\label{RTsuperclosep}
Assume that $\mathcal{T}_{h}$ satisfies the $(\alpha,\beta)$-condition. Then
\begin{equation*}
\|\Pi_{h}^1\bm{p}-\bm{p}_{h}^1\|\lesssim
h^{2+\min(\frac{1}{2},\alpha,\frac{\beta}{2})}\big(|\bm{p}|_{2,\infty,\Omega}+\|\bm{p}\|_{3,\Omega}\big).
\end{equation*}
\end{theorem}
\begin{proof}
For simplicity, the super-index $r=1$ is suppressed in the proof. Consider the discrete Helmholtz decomposition 
\begin{equation}\label{disHelmholtzxi}
\bm{\xi}_h:=\Pi_h\bm{p}-\bm{p}_h=\nabla^\perp w_{h}\oplus\grad_{h}v_{h},
\end{equation}
for some $\{v_{h},w_{h}\}\in\mathcal{V}^1_{h}\times\mathcal{W}^2_{h}$. 
Let $\bm{q}_h=\grad_h v_h/\|\grad_h v_h\|_{a}.$ By Lemma \ref{disPoincare} and Lemma \ref{superclosedivp},
\begin{equation}\label{bdgradv}
\begin{aligned}
&\|\grad_{h}v_{h}\|_{a}=(\grad_hv_h,\bm{q}_h)_a=-(v_h,\divg\bm{q}_h)\\
&=-\big(v_h,\frac{\divg\bm{\xi}_h}{\|\grad_h v_h\|_{a}}\big)\lesssim \|\divg\bm{\xi}_h\|_{0,\Omega}\lesssim h^{r+2}\|u\|_{r+2+\delta_{r0}}.
\end{aligned}
\end{equation}
It remains to bound $\nabla^\perp w_h$. Let $\bm{q}_h=\nabla^\perp w_h/\|\nabla^\perp w_h\|_{a}$. The orthogonality implies
\begin{equation}\label{gradrtotal}
\begin{aligned}
\|\nabla^\perp w_h\|_{a}&=-(a(\bm{p}-\Pi_{h}\bm{p}),\bm{q}_h)+(a(\bm{p}-\bm{p}_{h}),\bm{q}_h):=I+II.
\end{aligned}
\end{equation}
$I$ is split as
$$I=((\bar{a}-a)(\bm{p}-\Pi_{h}\bm{p}),\bm{q}_h)-(\bar{a}(\bm{p}-\Pi_{h}\bm{p}),\bm{q}_h).$$
By $\|\bar{a}-a\|_{0,\infty,\Omega}=O(h)$, \eqref{approxRT} and Lemma \ref{RT1mainlemma},
\begin{equation}\label{RTbdI}
\begin{aligned}
|I|\lesssim h^{3}|\bm{p}|_{2,\Omega}+h^{2+\min(\frac{1}{2},\alpha,\frac{\beta}{2})}\big(|\bm{p}|_{2,\infty,\Omega}+\|\bm{p}\|_{3,\Omega}\big).
\end{aligned}
\end{equation}
By $\divg\bm{q}_h=0$, $\|\bm{q}_h\|_{0,\Omega}\approx1$, \eqref{RTerreqn} and \eqref{supercloseu},
\begin{equation}\label{RTbdII1}
\begin{aligned}
II&=(\bm{q}_{h},\bm{b}(u-u_{h}))\\
&=(\bm{b}\cdot\bm{q}_{h},u-P_{h}u+P_{h}u-u_{h})\\
&=(\bm{b}\cdot\bm{q}_{h}-P_{h}(\bm{b}\cdot\bm{q}_{h}),u-P_{h}u)+O(h^3)
\|u\|_{3,\Omega}\\
&=O(h^{4})|\bm{b}\cdot\bm{q}_{h}|_{h,2,\Omega}|u|_{2,\Omega}
+O(h^{3})
\|u\|_{3,\Omega}.
\end{aligned}
\end{equation}
Since $\bm{q}_h|_T\in\CP_1(T)^2,$ the inverse estimate implies
$$|\bm{b}\cdot\bm{q}_h|_{2,T}\lesssim\|\bm{q}_h\|_{0,T}+\|D^1\bm{q}_h\|_{0,T}\lesssim h_T^{-1}\|\bm{q}_h\|_{0,T}.$$ 
\eqref{RTbdII1} then reduces to
\begin{equation}\label{RTbdII}
\begin{aligned}
II=O(h^3)
\|u\|_{3,\Omega}.
\end{aligned}
\end{equation}
Then the theorem follows from \eqref{bdgradv}--\eqref{RTbdI}, and \eqref{RTbdII}.
\qed\end{proof}

\section{Superconvergent recovery}\label{SR}
In this section, we introduce a new recovery operator $R_h^r:\mathcal{Q}_{h}^r\rightarrow\mathcal{W}_{h}^{r+1}\times\mathcal{W}_{h}^{r+1}$. For $\bm{q}_{h}\in\CQ_{h}^r,$ it suffices to specify nodal values of $R_h^r\bm{q}_{h}$. Here a node is the location of the degree of freedom of Lagrange elements, which can be a vertex of a triangle or an interior point of an edge/ triangle. For vertices $\bm{z}_1, \bm{z}_2, \bm{z}_3\in\Nh$, let $\overline{\bm{z}_1\bm{z}_2}$ denote the edge with endpoints $\bm{z}_1, \bm{z}_2$ and $\overline{\bm{z}_1\bm{z}_2\bm{z}_3}$ the triangle with vertices $\bm{z}_1, \bm{z}_2, \bm{z}_3$. $R_h^r$ is defined in three steps.

\emph{Step 1.} For each vertex $\bm{z}\in\mathcal{N}_h$, let $R_h^r\bm{q}_{h}(\bm{z}):=\bm{q}_z(\bm{z})$, where $\bm{q}_z\in\mathcal{P}_{r+1}(\omega_z)^{2}$ minimizes the quadratic functional
\begin{equation*}
\begin{aligned}
\mathcal{F}(\bm{q})&=\sum_{e\in\Eh(\omega_z)}\sum_{j=1}^{r+1}\big(\CN^j_{e}(\bm{q})-\mathcal{N}^j_{e}(\bm{q}_{h})\big)^2\\
&+\sum_{T\in\Th(\omega_z)}\sum_{l=1}^{r(r+1)/2}\sum_{m=1}^2\big(\mathcal{N}_T^{lm}(\bm{q})-\mathcal{N}_T^{lm}(\bm{q}_{h})\big)^2,
\end{aligned}
\end{equation*}
subject to $\bm{q}\in\mathcal{P}_{r+1}(\omega_z)^{2}$. 

\emph{Step 2.} For each node $\bm{z}$ in  the interior of an edge $e=\overline{\bm{z}_1\bm{z}_2}\in\Eh$, let 
$$R_h^r\bm{q}_{h}(\bm{z}):=(1-\alpha)\bm{q}_{z_1}(\bm{z})+\alpha\bm{q}_{z_2}(\bm{z}),\quad \alpha=|\bm{z}-\bm{z}_1|/|e|.$$ 

\emph{Step 3.} For each node $\bm{z}$ in the interior of the triangle $T=\overline{\bm{z}_1\bm{z}_2\bm{z}_3}\in\Th$, let $$R_h^r\bm{q}_{h}(\bm{z}):=\alpha_1\bm{q}_{z_1}(\bm{z})+\alpha_2\bm{q}_{z_2}(\bm{z})+\alpha_3\bm{q}_{z_3}(\bm{z}),$$ 
where $\alpha_1, \alpha_2, \alpha_3$ are barycentric coordinates of $\bm{z}$ w.r.t. $\bm{z}_1, \bm{z}_2,$ and $\bm{z}_3$.

In some cases, $\omega_z$ needs be enlarged to ensure that the above LS problem has a unique solution. Since $R_h^r$ depends only on the degrees of freedom of the $RT_r$ element, $R_h^r\bm{q}$ is well-defined for all $\bm{q}\in\mathcal{Q}$ and $R_h^r\Pi_h^r\bm{q}=R_h^r\bm{q}$. {\color{blue}Recall that
$\mathcal{N}_{e}^j(\bm{q})=\bm{q}(\bm{g}_j)\cdot\bm{n}_{e}$ if $\bm{q}\in\mathcal{P}_{r+1}(T)^2$ and $e\in\Eh(T)$.}

To clarify the recovery procedure, we give details to  two important cases: $RT_{0}$ and $RT_{1}$ elements.

\textbf{Example 1.} \emph{$RT_{0}$ elements on triangular meshes.} In this case, $R_h^0\bm{q}_{h}$ is a continuous piecewise linear function. At step $1$, let $\{e_{j}\}_{j=1}^J=\Eh(\omega_z).$ Let $\bm{m}_{j}=(m_{j1},m_{j2})^\intercal$ be the midpoint of $e_{j}$ and $\bm{n}_{j}=(n_{j1},n_{j2})^\intercal$ be a unit normal to $e_{j}$. Then $\bm{q}_{z}=(c_{1}+c_{2}x_1+c_{3}x_2, c_{4}+c_{5}x_1+c_{6}x_2)^\intercal\in\mathcal{P}_{1}(\omega_z)^2$ is the minimizer of 
\begin{equation*}
\begin{aligned}
&\mathcal{F}(\bm{q})=\sum_{j=1}^J\big(\bm{q}(\bm{m}_{j})\cdot\bm{n}_{j}-\bm{q}_{h}(\bm{m}_{j})\cdot\bm{n}_{j}\big)^{2},\\
&\text{subject to } \bm{q}\in\mathcal{P}_{1}(\omega_z)^{2}.
\end{aligned}
\end{equation*} 
Equivalently, $\bm{c}_z=(c_{1},\ldots,c_{6})^{\intercal}$ satisfies the normal equation $\bm{A}_z^\intercal\bm{A}_z\bm{c}_z=\bm{A}_z^\intercal\bm{d}_z$, where $\bm{d}_z=(\bm{q}_{h}(\bm{m}_{1})\cdot\bm{n}_{1},\ldots,\bm{q}_{h}(\bm{m}_J)\cdot\bm{n}_J)^\intercal$, $\bm{A}_z=(\bm{a}_{1}^\intercal,\ldots,\bm{a}_J^\intercal)^\intercal$ is an $N\times6$ matrix, $\bm{a}_{j}=(n_{j1}, m_{j1}n_{j1}, m_{j2}n_{j1}, n_{j2}, m_{j1}n_{j2}, m_{j2}n_{j2})$. Then $R_h\bm{q}_{h}(\bm{z})=\bm{q}_{z}(\bm{z})$ for $\bm{z}\in\CN_h.$

To avoid ill-conditioned $\bm{A}_z$ on graded meshes, we calculate $\bm{q}_z$ by scaling it properly. Let $h_z=|\omega_z|^\frac{1}{2}$ and $\hat{\bm{q}}_z(\hat{\bm{x}})=\bm{q}_z(\bm{z}+h_z\hat{\bm{x}})=(\hat{c}_{1}+\hat{c}_{2}\hat{x}_1+\hat{c}_{3}\hat{x}_2, \hat{c}_{4}+\hat{c}_{5}\hat{x}_1+\hat{c}_{6}\hat{x}_2)^\intercal$. Then
$\hat{\bm{c}}_z=(\hat{c}_{1},\ldots,\hat{c}_{6})^\intercal$ solves $\hat{\bm{A}}_z^\intercal\hat{\bm{A}}_z\hat{\bm{c}}_z=\hat{\bm{A}}_z^\intercal\bm{d}_z$, where $\hat{\bm{A}}_z=(\hat{\bm{a}}_{1}^\intercal,\ldots,\hat{\bm{a}}_J^\intercal)^\intercal$, $\hat{\bm{a}}_{j}=(n_{j1}, \hat{m}_{j1}n_{j1}, \hat{m}_{j2}n_{j1}, n_{j2}, \hat{m}_{j1}n_{j2}, \hat{m}_{j2}n_{j2}), \hat{\bm{m}}_{j}=(\bm{m}_{j}-\bm{z})/h_z=(\hat{m}_{j1},\hat{m}_{j2})$. Then $R_h^0\bm{q}_{h}(\bm{z})=(\hat{c}_{1},\hat{c}_{4})^\intercal$.

\textbf{Example 2.} \emph{$RT_{1}$ elements on triangular meshes.} In this case, $R_h^1\bm{q}_{h}$ is a continuous piecewise quadratic function. At step 1, let $\{e_{j}\}_{j=1}^J=\Eh(\omega_z)$ and $\{T_l\}_{l=1}^L=\Th(\omega_z)$. Let
\begin{equation*}
\bm{q}_z=\begin{pmatrix}c_{1}+c_{2}x_1+c_{3}x_2+c_{4}x_1^{2}+c_{5}x_1x_2+c_{6}x_2^{2}
\\c_{7}+c_{8}x_1+c_{9}x_2+c_{10}x_1^{2}+c_{11}x_1x_2+c_{12}x_2^{2}\end{pmatrix}\in\mathcal{P}_{2}(\omega_z)^{2}
\end{equation*}
minimize
\begin{equation*}
\begin{aligned}
\mathcal{F}(\bm{q})&=\sum_{j=1}^J\big(\bm{q}(\bm{x}_{j})\cdot\bm{n}_{j}
-\bm{q}_{h}(\bm{x}_{j})\cdot\bm{n}_{j}\big)^{2}+\big(\bm{q}(\bm{y}_{j})\cdot\bm{n}_{j}
-\bm{q}_{h}(\bm{y}_{j})\cdot\bm{n}_{j}\big)^{2}\\
&+\sum_{l=1}^L\sum_{m=1}^{2}\left(\frac{1}{|T_l|}\int_{T_l}q_m-\frac{1}{|T_l|}\int_{T_l}q_{h,m}\right)^{2},\quad\bm{q}\in\mathcal{P}_{2}(\omega_z)^{2},
\end{aligned}
\end{equation*} 
where $\bm{q}=(q_{1},q_{2})^{\intercal}, \bm{q}_{h}=(q_{h,1},q_{h,2})^{\intercal}$,
$\bm{x}_j=\frac{3+\sqrt{3}}{6}\bm{a}_{j}+\frac{3-\sqrt{3}}{6}\bm{b}_{j}$,
$\bm{y}_j=\frac{3-\sqrt{3}}{6}\bm{a}_{j}+\frac{3+\sqrt{3}}{6}\bm{b}_{j},$ and $e_{j}=\overline{\bm{a}_{j}\bm{b}_{j}}.$
Equivalently, $\bm{c}_z=(c_{1},\ldots,c_{12})^{\intercal}$ solves the normal equation $\bm{A}_z^\intercal\bm{A}_z\bm{c}_z=\bm{A}_z^\intercal\bm{d}_z$, where 
\begin{equation*}
\begin{aligned}
\bm{d}_z&=(\bm{q}_{h}(\bm{x}_{1})\cdot\bm{n}_{1},\bm{q}_{h}(\bm{y}_{1})\cdot\bm{n}_{1},\bm{q}_{h}(\bm{x}_{2})\cdot\bm{n}_{2},\bm{q}_{h}(\bm{y}_{2})\cdot\bm{n}_{2},\ldots,\\
&\quad\left.\bm{q}_{h}(\bm{y}_J)\cdot\bm{n}_J,\frac{1}{|T_1|}\int_{T_{1}}q_{h,1}, \frac{1}{|T_1|}\int_{T_{1}}q_{h,2}, \ldots, \frac{1}{|T_L|}\int_{T_L}q_{h,2}\right)^\intercal,
\end{aligned}
\end{equation*} 
and $\bm{A}_z=(\bm{a}_{1}^\intercal,\ldots,\bm{a}_{2J+2L}^\intercal)^\intercal$ is a $(2J+2L)\times12$ matrix, 
\begin{equation*}
\begin{aligned}
&\bm{a}_{2j-1}=(n_{j1}\bm{\xi}_{j}, n_{j2}\bm{\xi}_{j}),
\quad\bm{a}_{2j}=(n_{j1}\bm{\eta}_{j}, n_{j2}\bm{\eta}_{j}),\\
&\bm{\xi}_{j}=(1,x_{j1},x_{j2},x_{j1}^2,x_{j1}x_{j2},x_{j2}^2),\\
&\bm{\eta}_{j}=(1,y_{j1},y_{j2},y_{j1}^2,y_{j1}y_{j2},y_{j2}^2),
\quad1\leq j\leq J,\\
&\bm{a}_{2N+2l-1}=\frac{1}{|T_l|}\int_{T_l}(1,x_1,x_2,x_1^2,x_1x_2,x_2^{2},0,0,0,0,0,0),\\
&\bm{a}_{2N+2l}=\frac{1}{|T_l|}\int_{T_l}(0,0,0,0,0,0,1,x_1,x_2,x_1^2,x_1x_2,x_2^{2}),\quad1\leq l\leq L.
\end{aligned}
\end{equation*}
Then $R_h^1\bm{q}_{h}(\bm{z})=\bm{q}_z(\bm{z})$ for $\bm{z}\in\CN_h$. 
At step 2, for the midpoint $\bm{z}$ of the edge $e=\overline{\bm{z}_1\bm{z}_2}$, 
$R_h\bm{q}_{h}(\bm{z})=(\bm{q}_{z_1}(\bm{z})+\bm{q}_{z_2}(\bm{z}))/2$. 
one can again introduce the scaled polynomial $\hat{\bm{q}}_z(\hat{\bm{x}})=\bm{q}_z(\bm{z}+h_z\hat{\bm{x}})$ in practice.

Assume that the solution of each local LS problem at each vertex $\bm{z}$ is unique. By definition $R_h^r$ preserves $(r+1)$-degree polynomials, namely, $R_h^r\bm{q}=\bm{q}$ on $T$ for $\bm{q}\in\mathcal{P}_{r+1}(\omega_T)^{2}$, which leads to the super-approximation property
$\|\bm{q}-R_h^r\bm{q}\|_{0,\Omega}=O(h^{r+2}).$
However, it's not obvious that these local LS problems are uniquely solvable. The next obvious lemma gives several statements equivalent to uniqueness.
\begin{lemma}\label{TFAE}
The following statements are equivalent:
\begin{enumerate}
\item There exists a unique $\bm{q}_z$ at $\bm{z}.$ 
\item $A_z\bm{c}=\bm{0}$ implies $\bm{c}=\bm{0}$.
\item $\Pi_h^r\bm{q}_z=0$ on $\omega_z$ implies $\bm{q}_z\equiv0$. 
\end{enumerate}
\end{lemma}
Hence it suffices to study the unisolvence of $\Pi_h^r$ on $\mathcal{P}_{r+1}(\omega_z)^2$. $\Pi_h^r$ is  moment-based interpolation while nodal interpolation is often easier to analyze. The next lemma reduces Statement 3 in Lemma \ref{TFAE} to the case of Lagrange interpolation.

\begin{lemma}\label{moment2nodal}
Assume $\Pi_h^r\bm{q}_z=0$ on $\omega_z$. Then $\bm{q}_z=\nabla^\perp w$ for some
$w\in\mathcal{P}_{r+2}(\omega_z).$ In addition, for $e\in\Eh(\omega_z)$,
$w(\bm{l})=0$ at any Lobatto quadrature point $\bm{l}$ on $e$.
\end{lemma}
\begin{proof}
$\Pi_h^r\bm{q}_z=0$ and \eqref{commutative} imply
\begin{equation*}
\divg\bm{q}_z=\divg(\bm{q}_z-\Pi_h^r\bm{q}_z)
=\divg\bm{q}_z-P_h^r\divg\bm{q}_z=0.
\end{equation*}
Hence $\bm{q}_z=\nabla^\perp w$ for some $w\in\mathcal{P}_{r+2}(\omega_z)$. Given $e=\overline{\bm{a}\bm{b}}\in\Eh(\omega_z)$, 
\begin{equation*}
w(\bm{b}) - w(\bm{a})=\int_{e}\dte{w}=\int_{e}\bm{q}_z\cdot\bm{n}_{e}=\int_{e}\Pi_h^r\bm{q}_z\cdot\bm{n}_{e}=0.
\end{equation*}
Hence $w(\bm{z})\equiv c$ for all vertices $\bm{z}$ in $\omega_z$. By subtracting $c$ from $w$, we can assume that $w$ vanishes at all vertices. For $v\in\mathcal{P}_r(e)$, 
$$\int_{e}w\dte{v}=-\int_{e}v\dte{w}=-\int_{e}\bm{q}_z\cdot\bm{n}_{e}v=-\int_{e}\Pi_h^r\bm{q}_z\cdot\bm{n}_{e}v=0,\quad $$ and thus 
\begin{equation}\label{orthoe}
\int_{e}w\tilde{v}=0\quad \text{for all}\ \tilde{v}\in\mathcal{P}_{r-1}(e).
\end{equation}
Note that on $e=\overline{\bm{a}\bm{b}},$ the Lobatto quadrature
$\int_{e}f=\sum_{j=1}^{r+2}\mu_{j}f(\bm{l}_j)$
is exact for $f\in\mathcal{P}_{2r+1}(e)$, where $\bm{l}_j=\bm{a}+(\bm{b}-\bm{a})\hat{l}_{j},$ $\{\hat{l}_{j}\}_{j=1}^{r+2}$ are zeros of the polynomial $\frac{d^r}{ds^r}\left(s^{r+1}(1-s)^{r+1}\right)$ and $\{\mu_{j}\}_{j=1}^{r+2}$ are corresponding weights. Let $\tilde{v}$ be the polynomial which is $\mu_{j}^{-1}$ at $\bm{l}_{j}$ and $0$ at rest of the $(r-1)$ interior quadrature points $\{\bm{l}_{i}\}_{i=2, i\neq j}^{r+1}$ in \eqref{orthoe}. Then 
$w(\bm{l}_{j})=\int_{e}w\tilde{v}=0.$
The proof is complete.
\qed\end{proof}

\begin{figure}[tbhp]
\centering
\includegraphics[width=13.0cm,height=4.5cm]{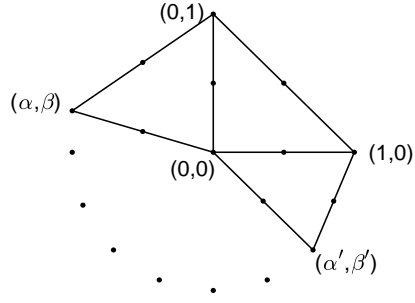}
\caption{A local patch containing the reference triangle.}
\label{patch}
\end{figure}
The next theorem gives practical criteria of checking the well-posedness of $R_h^0$ and $R_h^1$.
\begin{theorem}\label{RTuniqueness}
Let $\bm{z}$ be a vertex in $\Th$. If $\#\mathcal{T}(\omega_z)\geq5$ and the sum of each pair of adjacent angles in $\omega_z$ is $\leq\pi$, then there exists a unique $\bm{q}_z$ at $\bm{z}$ for $R_h^0$.
If $\#\mathcal{T}(\omega_z)\geq4$, then there exists a unique $\bm{q}_z$ at $\bm{z}$ for $R_h^1$.
\end{theorem}
\begin{proof}
Assume $\Pi_h^r\bm{q}_z=0$ on $\omega_z$.  By Lemma \ref{moment2nodal}, $\bm{q}_z=\nabla^\perp w$ for some $w\in\mathcal{P}_{r+2}(\omega_z)$.
If $r=0$, then $w\in\mathcal{P}_2(\omega_z)$ vanishes at all vertices in $\omega_z$ and thus $w=0$ by Theorem 2.3 in \cite{NZ2004}. Hence $\bm{q}_z=\bm{0}.$ 

If $r=1$, $w\in\CP_3(\omega_z)$ vanishes at all vertices and midpoints of edges in $\omega_z$. Without loss of generality, we can assume that $\bm{z}=(0,0)$ and the reference triangle $\hat{T}$ spanned by $(0,0), (0,1), (1,0)$ is in $\Th(\omega_z)$. 

If $w$ is reducible, then the zero set $w^{-1}(0)$ is the union of three straight lines(counting multiplicity) or the union of a straight line and a conic. Clearly three lines cannot pass all vertices and midpoints in $\omega_z$ provided $\#\Th(\omega_z)\geq4$. If $w^{-1}(0)$ contains a conic branch $C$, then $C$ must contain at least two vertices $\bm{a}, \bm{b}$ in $\omega_z$ because $\#\Th(\omega_z)\geq4$. However, $C$ cannot pass through $(\bm{a}+\bm{b})/2$ by elementary geometry.

Hence reducible $w$ cannot vanish at all nodes in $\omega_z$ and we can assume 
\begin{equation*}
w=c_{1}x_1^{3}+c_{2}x_1^{2}x_2+c_{3}x_1x_2^{2}+c_{4}x_2^{3}+c_{5}x_2^{2}+c_{6}x_1x_2+c_{7}x_2^{2}+c_{8}x_1+c_{9}x_2
\end{equation*}
is irreducible. Furthermore, we can assume one of the coefficients of highest order terms is $1$, say $c_{1}=1$(similar argument for $c_{2}, c_{3}$ or $c_{4}=1$). Let $(\alpha,\beta)$ be the vertex outside $\hat{T}$ next to $(0,1)$, see Fig.~\ref{patch}. Solving the linear system of equations
\begin{equation*}
\begin{aligned}
&w(1,0)=w(0,1)=w(1/2,0)=w(0,1/2)=w(1/2,1/2)\\
&=w(\alpha,\beta)=w\left(\frac{\alpha}{2},\frac{\beta+1}{2}\right)=w\left(\frac{\alpha}{2},\frac{\beta}{2}\right)=0,
\end{aligned}
\end{equation*}
we have
\begin{equation}\label{c23}
c_{1}=\frac{3-3\alpha}{1+\beta},\quad c_{2}=\frac{3\alpha(\alpha-1)}{\beta(1+\beta)}.
\end{equation}
Note that $\beta\neq0, \beta\neq-1$ in \eqref{c23}, otherwise the irreducible cubic curve $w^{-1}(0)$ intersects with a line at five distinct points, which is impossible by B\'ezout's theorem (see \cite{Shafa}). Also $\alpha\neq1$ otherwise it violates the topology of the patch $\omega_z$. Hence $\alpha/\beta=-c_{2}/c_{1}$. Let $(\alpha^{\prime}, \beta^{\prime})$ be the vertex outside $\hat{T}$ next to $(1,0)$. Similarly we have $\alpha^{\prime}/\beta^{\prime}=-c_{2}/c_{1}$. Then it forces $(\alpha,\beta)=(\alpha^{\prime},\beta^{\prime}),$ which contradicts $\#\Th(\omega_z)\geq4$. Hence $w\equiv0$ and $\bm{q}_z\equiv0.$

Therefore by Lemma \ref{TFAE}, there exists a unique $\bm{q}_z$ for $r=0, 1$.
\qed\end{proof}
We say a vertex $\bm{z}$ is good if the condition in Theorem \ref{RTuniqueness} holds at $\bm{z}$, otherwise it is a bad vertex. In practice, $\Th$ typically has a few bad vertices, e.g., boundary vertices. There are several ways of dealing with a bad vertex $\bm{z}$. If $\bm{z}$ is directly connected to a good vertex $\bm{z}^\prime$, one can define $\omega_{z}:=\omega_{z^\prime}$ and thus $\bm{A}_z$ is of full column rank. 
A more convenient way is to \emph{empirically} add some extra elements to the patch $\omega_z$ in practice, e.g., enlarge $\omega_z$ by one layer. Alternatively, one can solve a rank-deficient local least squares problem, which might reduce the rate of superconvergence of $R^r_h.$

In the rest of this paper, we assume that
$$\text{At each vertex $\bm{z}$, there exists a unique $\bm{q}_z$}.$$
Using the uniqueness of the LS solution, we obtain the boundedness of $R_h^r$.
\begin{theorem}\label{RTboundedness}
For $\bm{q}_h\in\CQ_h^r$ and $T\in\Th$, 
$$\|R_h^r\bm{q}\|_{0,T}\lesssim\|\bm{q}\|_{0,\omega_T},\quad r=0,1.$$
\end{theorem}
\begin{proof}
For $\bm{z}\in\CN_h,$ Let $\sigma_{\min}$ and $\sigma_{\max}$ be the minimum and maximum singular values of $\hat{\bm{A}}_z$ respectively. The goal is to show that $\sigma_{\min}$ is uniformly bounded away from $0$. MAC implies $\#\Th(\omega_z)\leq N_{\max}=2\pi/\Theta$. Hence it suffices to consider the case $\#\Th(\omega_z)=N$ for some fixed $N\leq N_{\max}$. In this case, $\#\Eh(\omega_z)=2N$. Let $N_1=2N, N_2=6$ provided $r=0$ and $N_1=6N, N_2=12$ provided $r=1$. Let $M_{N_1\times N_2}$ and $S_{N_1\times N_2}$ be the set of $N_1\times N_2$ matrices and $N_1\times N_2$ rank-deficient matrices, respectively. It is well known that $\sigma_{\min}=\text{dist}(\hat{\bm{A}}_z,S_{N_1\times N_2})$, the distance (measured by matrix $2$-norm) from $\hat{\bm{A}}_z$ to rank-deficient matrices. $\text{dist}(\cdot,S_{N_1\times N_2})$ is continuous on $M_{N_1\times N_2}$. Recall that $\hat{\bm{A}}_z$ is the scaled LS coefficient matrix determined by $\omega_z$. Consider all possible $\omega_z$ and define
$$\mathcal{A}_z=\{\hat{\bm{A}}_z\in M_{N_1\times N_2}: \#\Th(\omega_z)=N, \omega_z \text{ satisfies MAC} \}.$$
Clearly $\mathcal{A}_z$ is a compact set in $M_{N_1\times N_2}$ and any $\hat{\bm{A}}_z\in\mathcal{A}_z$ is of full rank by the uniqueness assumption. Hence $\sigma_{\min}=\text{dist}(\hat{\bm{A}}_z,S_{N_1\times N_2})\geq C_{1}>0$, where $C_{1}$ depends only on the minimum angle $\Theta$. The maximum singular value $\sigma_{\max}\leq C_{2}$, where $C_{2}$ only depends on $\Omega$. For $\bm{q}_{h}\in\mathcal{Q}_{h}^r$,
\begin{equation}\label{magc}
\begin{aligned}
|\hat{\bm{c}}_z|&\leq\|(\hat{\bm{A}}_z^{\intercal}\hat{\bm{A}}_z)^{-1}\|_{2}|\hat{\bm{A}}_z^{\intercal}\bm{d}_z|\leq\sigma_{\min}^{-2}\sigma_{\max}|\bm{d}_z|\\
&\leq C_{1}^{-2}C_{2}\|\bm{q}_{h}\|_{0,\infty,\omega_{z}}
\lesssim h_z^{-1}\|\bm{q}_{h}\|_{0,\omega_z},
\end{aligned}
\end{equation} 
where $|\cdot|$ is the Euclidean norm. 
Finally by \eqref{magc}, we have
\begin{equation*}
\|R_h^r\bm{q}_{h}\|_{0,T}\lesssim h\|R_h^r\bm{q}_{h}\|_{0,\infty,T}\lesssim h|\hat{\bm{c}}_z|\lesssim\|\bm{q}_{h}\|_{0,\omega_T},
\end{equation*}
which completes the proof.
\qed\end{proof}

The super-approximation property of $R_h$ follows from the uniqueness and boundedness results.
\begin{theorem}\label{RTsuperinterp}
For $\bm{q}\in H^{r+2}(\Omega)$, 
\begin{equation*}
\|\bm{q}-R_h^r\bm{q}\|_{0,\Omega}\lesssim h^{r+2}|\bm{q}|_{r+2,\Omega},\quad r=0, 1.
\end{equation*}
\end{theorem}
\begin{proof}
Let $T=\overline{\bm{z}_1\bm{z}_2\bm{z}_3}\in\Th$ and $T_1\subset\overline{\Omega}$ be a smallest local triangle containing $\omega_T$. Let $\bm{q}_{r+1}\in\mathcal{P}_{r+1}(T_1)^2$ be the degree-$(r+1)$ local Lagrange interpolant of $\bm{q}$ using based on $T_1$. By the uniqueness assumption, $R_h^r\bm{q}_{r+1}=\bm{q}_{r+1}$ on $T$.  It then follows from $R_h^r\Pi_h^r=R_h^r$ that 
\begin{equation}\label{Rh1}
\|\bm{q}-R_h^r\bm{q}\|_{0,T}\leq\|\bm{q}-\bm{q}_{r+1}\|_{0,T}+\|R_h^r\Pi_h^r(\bm{q}_{r+1}-\bm{q})\|_{0,T}.
\end{equation}
Using the boundedness from Theorem \ref{RTboundedness}, the stability in \eqref{Linfty}, and \eqref{feinterp},
\begin{equation}\label{Rh2}
\begin{aligned}
&\|R_h^r\Pi_h^r(\bm{q}_{r+1}-\bm{q})\|_{0,T}\lesssim\|\Pi_h^r(\bm{q}_{r+1}-\bm{q})\|_{0,\omega_T}\\
&\quad\lesssim h\|\Pi_h^r(\bm{q}_{r+1}-\bm{q})\|_{0,\infty,\omega_T}\lesssim h\|\bm{q}_{r+1}-\bm{q}\|_{0,\infty,\omega_T}\lesssim h^{r+2}|\bm{q}|_{r+2,T_1}.
\end{aligned}
\end{equation}
Combining \eqref{Rh1}, \eqref{Rh2} and the shape regularity $\Th$ completes the proof.
\qed\end{proof}
In the end, we present the superconvergent recovery estimate.
\begin{theorem}\label{RTsuper}
Assume that $\mathcal{T}_{h}$ satisfies the $(\alpha,\beta)$-condition. Then
\begin{equation*}
\|\bm{p}-R_h^r\bm{p}_{h}^r\|_{0,\Omega}\lesssim
h^{r+1+\min(\frac{1}{2},\alpha,\frac{\beta}{2})}\big(|\bm{p}|_{r+1,\infty,\Omega}+\|\bm{p}\|_{r+2,\Omega}\big),\quad r=0, 1.
\end{equation*}
\end{theorem}
\begin{proof}
The theorem follows from 
$$\|\bm{p}-R_h^r\bm{p}_{h}^r\|_{0,\Omega}\leq\|\bm{p}-R_h^r\bm{p}\|_{0,\Omega}+\|R_h^r(\Pi_h^r\bm{p}-\bm{p}_{h}^r)\|_{0,\Omega},$$
Theorems \ref{RTboundedness} and \ref{RTsuperinterp}, Theorem \ref{RTsuperclosep}($r=1$) or Theorem 4.5($r=0$) in \cite{YL2018}.
\qed\end{proof}

\section{Proof of Lemma \ref{RT1err2}}\label{Proof}
The following elementary triangular identities hold:
\begin{equation}\label{elem}
\begin{aligned}
&\cos\theta_{k}=(\ell_{k-1}^{2}+\ell_{k+1}^{2}-\ell_{k}^{2})/(2\ell_{k-1}\ell_{k+1}),
\quad\sin\theta_{k}=\ell_{k}/d,\quad d_k=\ell_{k-1}\ell_{k+1}/d,\\
&\bm{n}_{k-1}=-\sin\theta_{k+1}\bm{t}_{k}-\cos\theta_{k+1}\bm{n}_{k},\quad
\bm{n}_{k+1}=\sin\theta_{k-1}\bm{t}_{k}-\cos\theta_{k-1}\bm{n}_{k},\\
&\dtkm^2{}=\cos^{2}\theta_{k+1}\dtk^2{}-2\cos\theta_{k+1}\sin\theta_{k+1}\dtnk{}+\sin^{2}\theta_{k+1}\dnk^2{},\\
&\dtkp^2{}=\cos^{2}\theta_{k-1}\dtk^2{}+2\cos\theta_{k-1}\sin\theta_{k-1}\dtnk{}+\sin^{2}\theta_{k-1}\dnk^2{}.
\end{aligned}
\end{equation}
For each edge $e_k$, we define several associated geometric quantities $\{\alpha_{jl,k}^i\}_{1\leq i,j,l\leq2}$
\begin{equation*}
\begin{aligned}
&\alpha^1_{11,k}
=\frac{1}{24d\ell_{k}^{2}}\ell_{k-1}\ell_{k+1}\big(3\ell_{k}^{4}-(\ell_{k-1}^{2}-\ell_{k+1}^{2})^{2}\big),\\
&\alpha^1_{12,k}=\alpha^1_{21,k}=
\frac{1}{12d^{2}\ell_{k}}\ell_{k-1}^{2}\ell_{k+1}^{2}(\ell_{k-1}^{2}-\ell_{k+1}^{2}),
\quad\alpha^1_{22,k}=-\frac{1}{6d^{3}}\ell_{k+1}^{3}\ell_{k-1}^{3},\\
&\alpha^2_{11,k}
=\frac{1}{48\ell_{k}^{3}}(\ell_{k-1}^{2}-\ell_{k+1}^{2})\big(9\ell_{k}^{4}-(\ell_{k-1}^{2}-\ell_{k+1}^{2})^{2}\big),\\
&\alpha^2_{12,k}=\alpha^2_{21,k}
=-\alpha^1_{11,k},\quad
\alpha^2_{22,k}=-\alpha^1_{12,k}.
\end{aligned}
\end{equation*}
To prove Lemma \ref{RT1err2}, we introduce cubic bubble functions 
\begin{equation*}
\psi_{0}=\lambda_{1}\lambda_{2}\lambda_{3},\quad\psi_{k}=\lambda_{k-1}\lambda_{k+1}(\lambda_{k-1}-\lambda_{k+1}),\quad1\leq k\leq3.
\end{equation*} 
By counting the dimension, it is clear that $\{\psi_{k}\}_{k=0}^{3}$ can span polynomials in $\mathcal{P}_{3}(T)$ that vanish at $\{\bm{z}_k\}_{k=1}^3$ and midpoints of $\{e_{k}\}_{k=1}^{3}$. In fact, $\{\psi_{k}\}_{k=0}^{3}$ has been used to derive superconvergence of quadratic Lagrange elements (cf.\cite{HuangXu2008}) and a posteriori error estimators (cf.\cite{BaXuZheng2007}). 
\begin{lemma}\label{RT1err1}
For $\bm{p}_{2}\in\mathcal{P}_{2}(T)^{2}$,
\begin{equation*}
\bm{p}_{2}-\Pi_h^1\bm{p}_{2}=\nabla^\perp w,
\end{equation*}
where
\begin{equation*}
w=\alpha^i_{jl,\beta}\DD_{i,\beta}^{jl}(\bm{p}_{2})\psi_{0}
+\sum_{k=1}^{3}\frac{\ell_k^3}{12}\DD_{2,k}^{11}(\bm{p}_{2})\psi_k,\quad\forall1\leq \beta\leq3.
\end{equation*}
\end{lemma}
\begin{proof}
By $\Pi_h^1(\bm{p}_{2}-\Pi_h^1\bm{p}_{2})=0$ and using Lemma \ref{moment2nodal}, we have
\begin{equation}\label{ppq}
\bm{p}_{2}-\Pi_h^1\bm{p}_{2}=\nabla^\perp\big(\sum_{k=0}^{3}c_k\psi_k\big).
\end{equation}
For a unit vector $\bm{d}$ and the directional derivative $\partial_{\bm{d}}$, the definition of $\mathcal{RT}_{1}(T)$ 
implies that $\dd^2{\Pi_h^1\bm{p}_{2}}$ is proportional to $\bm{d}$. Then applying $\bm{d}^\perp\cdot\dd^2$ to \eqref{ppq} gives
\begin{equation}\label{generald}
\bm{d}^\perp\cdot\dd^2{\bm{p}_{2}}=\sum_{k=0}^{3}c_k\dd^3\psi_k.
\end{equation}
By direct calculation,
\begin{subequations}\label{3rddd}
\begin{align}
&\partial_{\bm{d}}^{3}\psi_{0}=6\partial_{\bm{d}}\lambda_{1}\partial_{\bm{d}}\lambda_{2}\partial_{\bm{d}}\lambda_{3},\\
&\partial_{\bm{d}}^{3}\psi_k=6\partial_{\bm{d}}\lambda_{k-1}\partial_{\bm{d}}\lambda_{k+1}(\partial_{\bm{d}}\lambda_{k-1}-\partial_{\bm{d}}\lambda_{k+1}),
\quad1\leq k\leq3.
\end{align}
\end{subequations}
In particular, $\dtk^3\psi_0=0$ and $\dtk^3\psi_j=-12\delta_{jk}/\ell_k^3.$
By \eqref{3rddd} and \eqref{generald} with $\bm{d}=\bm{t}_k$, we have 
\begin{equation}\label{alphak}
c_{k}=\frac{\ell_{k}^{3}}{12}\bm{n}_{k}\cdot\dtk^2{\bm{p}_{2}}=\frac{\ell_{k}^{3}}{12}\DD_{2,k}^{11}(\bm{p}_{2}),\quad1\leq k\leq 3.
\end{equation}
It remains to determine $c_{0}$. \eqref{generald} with $\bm{d}=\bm{n}_{k}$ implies that
\begin{equation}\label{alpha01st}
\DD_{1,k}^{22}(\bm{p}_{2})=c_{0}\dnk^3{\psi_{0}}+c_{k}\dnk^3{\psi_{k}}
+c_{k-1}\dnk^3{\psi_{k-1}}+c_{k+1}\dnk^3{\psi_{k+1}},
\end{equation}
By 
$\dnk{\lambda_{k}}=-1/d_{k}, \dnk{\lambda_{k+1}}=\cos\theta_{k-1}/d_{k+1}, 
\dnk{\lambda_{k-1}}=\cos\theta_{k+1}/d_{k-1},$
\eqref{3rddd} with $\bm{d}=\bm{n}_k$, \eqref{alphak}, and 
\eqref{alpha01st}, we obtain
\begin{equation}\label{intermediate}
\begin{aligned}
c_{0}&=-\frac{d_{k-1}d_{k}d_{k+1}}{6\cos\theta_{k-1}\cos\theta_{k+1}}\DD_{1,k}^{22}(\bm{p}_{2})\\
&+\frac{\ell_{k}^{3}}{12}d_{k}
\left(\frac{\cos\theta_{k+1}}{d_{k-1}}-\frac{\cos\theta_{k-1}}{d_{k+1}}\right)\DD_{2,k}^{11}(\bm{p}_{2})\\
&-\frac{\ell_{k-1}^{3}}{12}\frac{d_{k-1}}{\cos\theta_{k+1}}
\left(\frac{1}{d_{k}}+\frac{\cos\theta_{k-1}}{d_{k+1}}\right)\DD_{2,k-1}^{11}(\bm{p}_{2})\\
&+\frac{\ell_{k+1}^{3}}{12}\frac{d_{k+1}}{\cos\theta_{k-1}}
\left(\frac{1}{d_{k}}+\frac{\cos\theta_{k+1}}{d_{k-1}}\right)\DD_{2,k+1}^{11}(\bm{p}_{2}).
\end{aligned}
\end{equation}
Then using \eqref{elem} and \eqref{intermediate}, we obtain 
$c_{0}=\alpha^i_{jl,k}\DD_{i,k}^{jl}(\bm{p}_{2}), 1\leq k\leq3.$
\qed\end{proof}

Now we can prove Lemma \ref{RT1err2}. In the proof, we shall use the integral formula
\begin{equation}\label{intbary}
\int_{T}\lambda_1^{m_1}\lambda_2^{m_2}\lambda_3^{m_3}=\frac{2|T|m_1!m_2!m_3!}{(m_1+m_2+m_3+2)!},\quad\int_e\lambda_1^{m_1}\lambda_2^{m_2}=\frac{|e|m_1!m_2!}{(m_1+m_2+1)!},
\end{equation}
where $\lambda_1, \lambda_2$ are barycentric coordinates w.r.t.~the edge $e$.
\begin{proof}
Using \eqref{IP} and Lemma \ref{RT1err1}, we have
\begin{equation}\label{IntegrationByParts}
\begin{aligned}
\int_T(\bm{p}_{2}-\Pi_h^1\bm{p}_{2})\cdot\nabla^\perp w_{2}
=\sum_{k=1}^3\int_{e_{k}}w\nabla^\perp w_2\cdot\bm{t}_k
-\int_Tw\Delta w_{2}:=I+II.
\end{aligned}
\end{equation}
Recall that $\phi_k=\lambda_{k-1}\lambda_{k+1}$ and let $I_h$ be the linear interpolation.
Then using the hierarchical representation
\begin{equation}\label{hierarchy}
w_{2}-I_hw_{2}=-\frac{1}{2}\sum_{k=1}^{3}\ell_{k}^{2}\phi_k\dtk^2{w_2},
\end{equation}
and $\Delta\phi_k=2\nabla\lambda_{k-1}\cdot\nabla\lambda_{k+1}=-2\cos\theta_k/(d_{k-1}d_{k+1})$,
we obtain
\begin{equation}\label{Laplacian}
\Delta w_2=\frac{1}{4|T|^{2}}\sum_{k=1}^{3}\ell_{k}^{2}\ell_{k-1}
\ell_{k+1}\cos\theta_{k}\dtk^2w_2.
\end{equation}

It then follows from Lemma \ref{RT1err1}, \eqref{Laplacian}, and $\int_T\psi_{0}=|T|/60,$ $\int_T\psi_{k}=0, 1\leq k\leq 3,$ that
\begin{equation}\label{RT12ndterm}
\begin{aligned}
&II=-\frac{|T|}{60}c_0\Delta w_{2}=-\frac{1}{240|T|}\sum_{k=1}^{3}c_{0}\ell_{k}^{2}\ell_{k-1}\ell_{k+1}\cos\theta_k\dtk^2{w_{2}}\\
&\qquad=-\frac{1}{120}\sum_{k=1}^{3}\int_{e_{k}}\alpha^i_{jl,k}\DD_{i,k}^{jl}(\bm{p}_{2})\ell_{k}\cot\theta_k\dtk^2{w_{2}}.
\end{aligned}
\end{equation}
By the elementary identity
$\bm{t}_k=\frac{\cos\theta_{k+1}}{\sin\theta_k}\bm{n}_{k+1}-\frac{\cos\theta_{k-1}}{\sin\theta_k}\bm{n}_{k-1},$
Lemma \ref{RT1err1},
and $\psi_{k}=-\ell_k\dtk{(\phi_k^2)}/2$, we have
\begin{equation}\label{1st}
\begin{aligned}
I&=-\sum_{k=1}^{3}\frac{1}{12}\int_{e_{k}}\ell_{k}^{3}\DD_{2,k}^{11}(\bm{p}_{2})
\psi_{k}\nabla^\perp w_{2}\cdot\left(\frac{\cos\theta_{k-1}}{\sin\theta_{k}}\bm{n}_{k-1}
-\frac{\cos\theta_{k+1}}{\sin\theta_{k}}\bm{n}_{k+1}\right)\\
&=\sum_{k=1}^{3}\frac{1}{24}\int_{e_{k}}\ell_{k}^4\DD_{2,k}^{11}(\bm{p}_{2})
\phi^2_{k}\left(\frac{\cos\theta_{k-1}}{\sin\theta_{k}}\partial_{\bm{t}_{k}\bm{t}_{k-1}}^2w_2
-\frac{\cos\theta_{k+1}}{\sin\theta_{k}}\partial_{\bm{t}_{k}\bm{t}_{k+1}}^2w_2\right).
\end{aligned}
\end{equation}
Then using the quadrature rule \eqref{intbary},
\begin{equation*}
\begin{aligned}
I=\frac{1}{720}\sum_{k=1}^{3}\ell_{k}^5\DD_{2,k}^{11}(\bm{p}_{2})\left(\frac{\cos\theta_{k-1}}{\sin\theta_{k}}\partial_{\bm{t}_{k}\bm{t}_{k-1}}^2w_2-\frac{\cos\theta_{k+1}}{\sin\theta_{k}}\partial_{\bm{t}_{k}\bm{t}_{k+1}}^2w_2
\right).
\end{aligned}
\end{equation*}
In addition, \eqref{hierarchy} gives
\begin{equation*}
\begin{aligned}
\partial_{\bm{t}_{k}\bm{t}_{k-1}}^{2}w_{2}&=-\frac{\ell_{k}}{2\ell_{k-1}}\dtk^2{w_{2}}
+\frac{\ell_{k+1}^2}{2\ell_{k-1}\ell_{k}}\dtkp^2{w_{2}}-\frac{\ell_{k-1}}{2\ell_{k}}\dtkm^2{w_{2}},\\
\partial_{\bm{t}_{k}\bm{t}_{k+1}}^{2}w_{2}&=-\frac{\ell_{k}}{2\ell_{k+1}}\dtk^2{w_{2}}
-\frac{\ell_{k+1}}{2\ell_{k}}\dtkp^2{w_{2}}+\frac{\ell_{k-1}^2}{2\ell_{k}\ell_{k+1}}\dtkm^2{w_{2}}.
\end{aligned}
\end{equation*}
Therefore,
\begin{equation}\label{RT11stterm}
\begin{aligned}
I&=\frac{1}{1440}\sum_{k=1}^{3}\int_{e_{k}}\left\{\frac{\ell_{k}^5}{\sin\theta_k}\DD_{2,k}^{11}(\bm{p}_{2})\left(\frac{\cos\theta_{k+1}}{\ell_{k+1}}-\frac{\cos\theta_{k-1}}{\ell_{k-1}}\right)\right.\\
&\qquad+\frac{\ell_{k-1}^4}{\sin\theta_{k-1}}\DD_{2,k-1}^{11}(\bm{p}_{2})\left(\cos\theta_{k}+\frac{\ell_{k}}{\ell_{k+1}}\cos\theta_{k+1}\right)\\
&\qquad-\left.\frac{\ell_{k+1}^4}{\sin\theta_{k+1}}\DD_{2,k+1}^{11}(\bm{p}_{2})\left(\frac{\ell_{k}}{\ell_{k-1}}\cos\theta_{k-1}+\cos\theta_{k}\right)\right\}\dtk^2w_2
\end{aligned}
\end{equation}
Combining \eqref{IntegrationByParts}, \eqref{RT12ndterm}, \eqref{RT11stterm} and using \eqref{elem}, we obtain Lemma \ref{RT1err2}.
\qed\end{proof}

\section{Numerical experiments}\label{sec6}
\begin{figure}[tbhp]
\centering
\includegraphics[width=12cm,height=5cm]{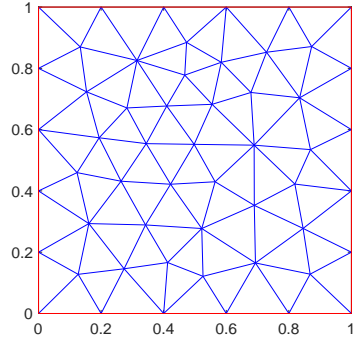}
\caption{Delaunay initial grid on a square.}
\label{initialmesh}
\end{figure}

\begin{figure}[tbhp]
\centering
\includegraphics[width=12cm,height=5.0cm]{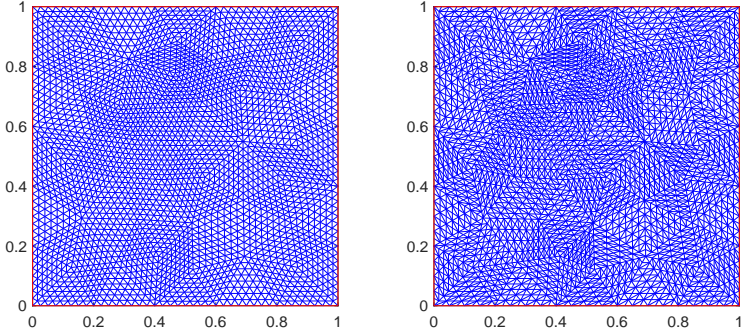}
\caption{(left)Regular refinement, 5504 elements. (right)Newest vertex bisection, 5504 elements.}
\label{refinement}
\end{figure}
We test our recovery operators $R_h^r$ with $r=1,2,3$ by the Poisson equation
\begin{equation*}
-\Delta u=f\text{ in }\Omega,
\end{equation*}
where $\Omega$ and $u$ will be given in the next three experiments. Readers are referred to \cite{YL2018} for numerical results on recovery superconvergence of the $RT_0$ element. 
The experiments are implemented using the iFEM package \cite{iFEM} in Matlab 2018b. In tables, $\|\cdot\|$ is the $L^2$-norm $\|\cdot\|_{0,\Omega}$, `nt' denotes the number of triangles. The order of convergence is $p$ such that error$\approx$ ndof$^{-\frac{p}{2}}$, where ndof is the number of degrees of freedom. The value of $p$ is computed by least squares. 

\textbf{Problem 1.} In the first experiment, let $\Omega$ be the unit square $[0,1]^2$ and  $$u=\exp(x_1+x_2)\sin(2\pi x_1)\sin(\pi x_2)$$ be the exact solution. We test the performance of $R_h^1$. {Due to Theorem \ref{RTuniqueness}, we do not enlarge the patch $\omega_z$ when $z$ is an interior vertex. If $z$ is a boundary vertex, extra neighboring elements are added to $\omega_z$ such that $\#\omega_z\geq8.$ It turns out that all local least squares problems are uniquely solvable.}
We start with the Delaunay triangulation in Fig.~\ref{initialmesh}, and
computed a sequence of meshes by regular refinement, i.e., dividing an element into four similar subelements by connecting the midpoints of each edge, see Table \ref{RT1regular}. We also computed a sequence of meshes by newest vertex bisection (cf.~\cite{Mitchell1990,iFEM}), see Fig.~\ref{refinement} and Table \ref{RT1bisection}. 

For regular refinement, the sequence of grids satisfies $(\alpha,\beta)$-condition with $(\alpha,\beta)=(\infty,1)$. For $RT_{1}$ elements, Theorem \ref{RTsuperclosep} predicts that 
$\|\Pi_h^1\bm{p}-\bm{p}_{h}^1\|=O(h^{2.5})$, which is confirmed by Table \ref{RT1regular}. In view of the high order recovery superconvergence  $\|\bm{p}-R_h^1\bm{p}_{h}^1\|=O(h^{3.4})$, our supercloseness estimate  $\|\bm{p}-R_h^1\bm{p}_{h}^1\|=O(h^{2.5})$ in Theorem \ref{RTsuper} may be suboptimal. 

The sequence of grids created by newest vertex bisection is far from uniformly parallel, i.e., almost no pair of adjacent triangles forms an $O(h^{1+\alpha})$ approximate parallelogram with some positive $\alpha$. Hence there is no supercloseness in Table \ref{RT1bisection}. Surprisingly, we still observe apparent superconvergence for $\|\bm{p}-R_h^1\bm{p}_{h}^1\|$.\\ 

{\textbf{Problem 2:} Although our supercloseness estimates only work for $RT_0$ and $RT_1$ elements, we perform numerical experiments on the recovery operators $R_h^2$ and $R_h^3$ for $RT_2$ and $RT_3$ elements. We use the same $\Omega, u,$ and initial mesh with regular refinement in Problem 1. Local patches $\omega_z$ is chosen in the same way as in Problem 1.
The numerical results are presented in Tables \ref{RT2regular} and \ref{RT3regular}. 

As mentioned in Problem 1, the sequence of grids satisfies $(\alpha,\beta)$-condition with $(\alpha,\beta)=(\infty,1)$. Unlike $RT_0$ and $RT_1$ elements, there is no supercloseness phenomenon for $RT_2$ and $RT_3$ even on regularly refined meshes. However, it can be observed that the rate of recovery superconvergence is at least $\|\bm{p}-R_h^r\bm{p}_h^r\|=O(h^{r+2})$ with $r=2,3$. Therefore, the supercloseness estimate is not a necessary ingredient of superconvergence analysis. We conjecture that the superconvergence is due to a large number of locally symmetric patches, see \cite{SSW1996} for the theory of Lagrange elements.}

\textbf{Problem 3.} Postprocessing superconvergence is often used to develop recovery-type a posteriori error estimator and adaptive FEMs.  
In the end, we test the adaptivity performance of $R_h^1$ on the domain $\Omega=[-1,1]^2\backslash \Omega_0$, where $\Omega_0$ is a right triangle whose smallest angle is $\omega=\pi/24$, see Fig.~\ref{meshadapt}(left). Let
$$u(r,\theta)=r^{\frac{\pi}{2\pi-\omega}}\sin\left(\frac{\pi}{2\pi-\omega}\theta\right)-\frac{r^{2}}{4},$$ 
where $(r,\theta)$ is the polar coordinate. The corresponding source $f=-\Delta u=1$. We use the classical adaptive feedback loop (cf.~\cite{Dorfler1996,MNS2000})
$$\textsf{SOLVE}\rightarrow\textsf{ESTIMATE}\rightarrow\textsf{MARK}
\rightarrow\textsf{REFINE}.$$ 
{It will return a sequence of meshes $\{\mathcal{T}_{h_\ell}\}_{\ell\geq0}$ and numerical solutions $\{\bm{p}_{h_\ell}\}_{\ell\geq0}$. The algorithm starts from the initial grid $\mathcal{T}_{h_0}$ in Fig.~\ref{meshadapt}(left). In the procedure \textsf{ESTIMATE}, $\eta_{\ell,T}=\|R_{h_\ell}^1\bm{p}_{h_\ell}^1-\bm{p}_{h_\ell}^1\|_{0,T}$ serves as a posteriori error estimator on each triangle $T\in\mathcal{T}_{h_\ell}$. The procedure \textsf{MARK} selects a collection of triangles $\mathcal{M}_\ell\subset\mathcal{T}_{h_\ell}$ such that $$\sum_{T\in\mathcal{M}_\ell}\eta_{\ell,T}^2\geq0.3\sum_{T\in\mathcal{T}_{h_\ell}}\eta_{\ell,T}^2.$$
Then the elements in $\mathcal{M}_\ell$ and necessary neighboring elements are refined by local mesh refinement strategy to yield a conforming subtriangulation $\mathcal{T}_{h_{\ell+1}} $ of $\mathcal{T}_{h_{\ell}} $. In particular, we use regular refinement with bisection closure in the procedure \textsf{REFINE}, see Fig.~\ref{meshadapt}(right) for an adaptively  refined triangulation. The numerical results are presented in Fig.~\ref{RT1}. } 

It can be observed that the adaptive algorithm yields optimal rate of convergence and apparent recovery superconvergence. {A distinct feature of the a posteriori error estimator $\eta_{h_\ell}:=\|R_{h_\ell}^1\bm{p}_{h_\ell}^1-\bm{p}_{h_\ell}^1\|_{0,\Omega}=\big(\sum_{T\in\mathcal{T}_{h_\ell}}\eta_{\ell,T}^2\big)^\frac{1}{2}$ is the well-known asymptotic exactness:
\begin{align*}
    \lim_{\ell\rightarrow\infty}\frac{\eta_{h_\ell}}{\|\bm{p}-\bm{p}^1_{h_\ell}\|_{0,\Omega}}=1,
\end{align*}
which can be numerically confirmed using the rates of superconvergence in Fig.~\ref{RT1} with a triangle inequality, see, e.g., \cite{BX2003b,XZ2003} for details.}

\begin{table}[tbhp]
\caption{$RT_{1}$ with regular refinement}
\centering
\begin{tabular}{|c|c|c|c|}
\hline
nt & $\|\bm{p}-\bm{p}_{h}^1\|$
 &$\|\Pi_h^1\bm{p}-\bm{p}_{h}^1\|$
&  $\|\bm{p}-R_h^1\bm{p}_{h}^1\|$ \\
\hline
             86   &3.176e-1	&4.297e-2	&5.186e-1\\
             344     &8.000e-2	&7.852e-3	&5.560e-2\\
             1376     &2.006e-2	&1.397e-3	&5.344e-3\\
             5504      &5.022e-3	&2.461e-4	&4.929e-4\\
             22106       &1.256e-3	&4.336e-5	&4.616e-5\\
\hline
order &1.998&2.501&3.414\\
\hline
\end{tabular}
\label{RT1regular}
\end{table}

\begin{table}[tbhp]
\caption{$RT_{1}$ with bisection refinement}
\centering
\begin{tabular}{|c|c|c|c|}
\hline
nt & $\|\bm{p}-\bm{p}_{h}^1\|$
 &$\|\Pi_h^1\bm{p}-\bm{p}_{h}^1\|$
&  $\|\bm{p}-R_h^1\bm{p}_{h}^1\|$ \\
\hline
             86   &3.176e-1	&4.297e-2	&5.186e-1\\
             344    &1.325e-1	&1.092e-1	&7.453e-2\\
             1376      &3.401e-2	&2.682e-2	&1.005e-2\\
             5504      &8.604e-3	&6.607e-3	&1.610e-3\\
             
             22016          &2.164e-3	&1.637e-3	&3.336e-4\\
\hline
             order   &1.979&2.020&2.605\\
\hline
\end{tabular}
\label{RT1bisection}
\end{table}

\begin{table}[tbhp]
\caption{$RT_2$ with regular refinement}
\centering
\begin{tabular}{|c|c|c|c|}
\hline
nt & $\|\bm{p}-\bm{p}_{h}^2\|$
 &$\|\Pi_h^2\bm{p}-\bm{p}_{h}^2\|$
&  $\|\bm{p}-R_h^2\bm{p}_{h}^2\|$ \\
\hline
             86   &2.378e-2	&5.201e-3	&1.505e-1\\
             344     &3.022e-3	&5.488e-4	&1.005e-2\\
             1376     &3.792e-4	&6.501e-5	&5.247e-4\\
             5504      &4.745e-5	&8.002e-6	&2.551e-5\\
             22106       &5.933e-6	&9.953e-7	&1.351e-6\\
\hline
order &2.993&3.080&4.215\\
\hline
\end{tabular}
\label{RT2regular}
\end{table}

\begin{table}[tbhp]
\caption{$RT_3$ with regular refinement}
\centering
\begin{tabular}{|c|c|c|c|}
\hline
nt & $\|\bm{p}-\bm{p}_{h}^3\|$
 &$\|\Pi_h^3\bm{p}-\bm{p}_{h}^3\|$
&  $\|\bm{p}-R_h^3\bm{p}_{h}^3\|$ \\
\hline
             86   &3.733e-3	&3.394e-3	&4.022e-2\\
             344     &2.359e-4	&2.140e-4	&1.180e-3\\
             1376     &1.478e-5	&1.338e-5	&2.668e-5\\
             5504      &9.242e-7	&8.354e-7	&6.377e-7\\
             22106      &5.777e-8	&5.217e-8&	2.116e-8\\
\hline
order &3.995&3.998&5.257\\
\hline
\end{tabular}
\label{RT3regular}
\end{table}

\begin{figure}[tbhp]
\centering
\includegraphics[width=12cm,height=5.0cm]{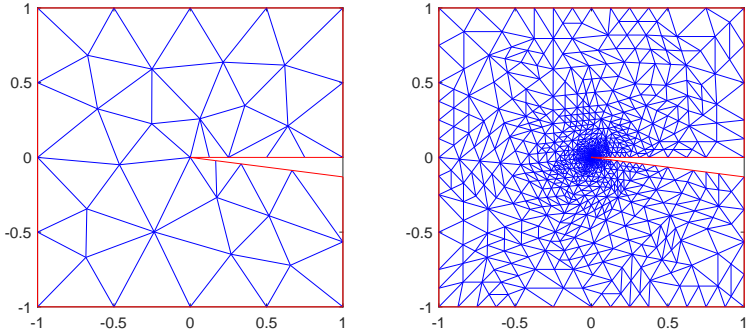}
\caption{(left)Initial grid for the adaptive algorithm. (right)Adaptive grid, 2026 elements.}
\label{meshadapt}
\end{figure}

\begin{figure}[tbhp]
\centering
\includegraphics[width=8cm,height=5cm]{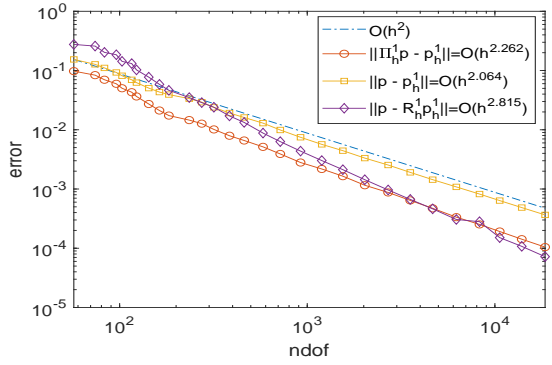}
\caption{Error curves for $RT_{1}$.}
\label{RT1}
\end{figure}

\section{Concluding remarks}
{In this paper, we develop supercloseness estimate for the second lowest order RT element and a family of postprocessing operators $R_h^r$ for higher order RT elements applied to second order elliptic equations. Since both the analysis of supercloseness and postprocessing operators are local, our superconvergence results can be adapted to Neumann and mixed boundary conditions. 
In practice, $R_h^r$ can be extended to 3-dimensional RT elements in a straightforward way although the theoretical analysis in this paper may need significant modifications, e.g., the supercloseness estimate and well-posedness of the local least squares problem would be more complicated. Readers are  also referred to \cite{DK1998} for numerical experiments on a different postprocessing operator for the lowest order RT elements in $\mathbb{R}^3$.}

\end{document}